\documentclass[11pt]{article}
\usepackage{amsmath}
\usepackage{hyperref}
\usepackage{pdfsync}
\usepackage{dsfont}
\usepackage{color}
\usepackage{esint}
\usepackage{amsthm}
\usepackage{enumerate}
\usepackage{amsfonts}
\usepackage{amssymb}
\usepackage{mathtools}
\usepackage{braket}
\usepackage{bm}
\usepackage{amsmath}
\usepackage{graphicx}
\usepackage{caption}
\usepackage{fullpage}
\usepackage{todonotes}
\usepackage{subcaption}

\usepackage[font=footnotesize,labelfont=bf]{caption}

\usepackage{mathtools}
\numberwithin{equation}{section}

\usepackage{amsfonts}
\usepackage{graphicx}
\providecommand{\U}[1]{\protect\rule{.1in}{.1in}}
\newtheorem{theorem}{Theorem}[section]

\newtheorem{corollary}[theorem]{Corollary}

\newtheorem{definition}[theorem]{Definition}
\newtheorem{example}[theorem]{Example}

\newtheorem{lemma}[theorem]{Lemma}

\newtheorem{proposition}[theorem]{Proposition}
\newtheorem{remark}[theorem]{Remark}

\allowdisplaybreaks[3]

\newcommand{\R}{\mathbb{R}}

\DeclareMathOperator{\sgn}{sgn}

\DeclareMathOperator{\argmin}{argmin}

\title{Tukey Depths and Hamilton-Jacobi Differential Equations}
\author{Martin Molina-Fructuoso and Ryan Murray}
\begin{document}
\maketitle

\begin{abstract}
  Widespread application of modern machine learning has increased the need for robust statistical algorithms. This work studies one such fundamental statistical measure known as the Tukey depth. We study the problem in the continuum (population) limit. In particular, we derive the associated necessary conditions, which take the form of a first-order partial differential equation. We discuss the classical interpretation of this necessary condition as the viscosity solution of a Hamilton-Jacobi equation, but with a non-classical Hamiltonian with discontinuous dependence on the gradient at zero. We prove that this equation possesses a unique viscosity solution, and that this solution always bounds the Tukey depth from below. In certain cases we prove that the Tukey depth is equal to the viscosity solution, and we give some illustrations of standard numerical methods from the optimal control community which deal directly with the partial differential equation. We conclude by outlining several promising research directions both in terms of new numerical algorithms and theoretical challenges.
\end{abstract}

\section{Introduction}
The success of statistical learning algorithms in recent years has often been tempered by a lack of understanding of their robustness. Questions of an algorithm's robustness or reliability is essential as we embed them into more parts of society, and is also intimately related to transparency, privacy, and fairness. These questions have sparked significant research activity and advances, and have significantly broadened the power and applicability of learning algorithms.

Classical statistics has, of course, invested significant effort in studying robustness. The most basic example of a robust statistic in one dimension is the \emph{median}. The median is robust in the sense that one has to move a significant proportion of data points in order to appreciably change the median. Similar arguments hold for the quantiles and quantile depths, namely the functions $Q(t) = \mu((-\infty,t))$ and $D(t) = \min ( \mu( (-\infty,t)),\mu(t,\infty) )$, where $\mu$ is an absolutely continuous measure on $\R$. These quantities are easily interpretable, and are robust to noise, outliers, and adversarial modification, all characteristics which are desirable for statistical learning algorithms.

There are many different ways to extend quantiles and depths to higher dimensions, and this work considers one such extension. Specifically, we consider the classical \emph{Halfspace} or \emph{Tukey depth}, which is defined by
\[
  T(x) := \inf\{ \mu(H) : H \text{ is a closed halfspace containing } x\}.
\]
This depth function is perhaps one of the most natural generalizations of one dimensional depths, and serves as an important prototype for many more sophisticated algorithms and ideas, both in the context of depth functions and more generally in robust statistical learning theory.

These depth functions have myriad applications in both pure and applied problems. For example, depth functions are used for 
robust treatment of outliers in estimation \cite{zuo2003projection} and classification \cite{jornsten2004clustering}, providing non-parametric statistical measures \cite{wang2005nonparametric}, creating higher-dimensional box plots \cite{rousseeuw1999bagplot}, and in functional data analysis \cite{cuevas2007robust}. They can be used to define generalized medians \cite{zuo2000general}, which provide a highly robust summary statistic in higher dimensions. Finally, they are objects of study in convex geometry, with applications to sampling theory, differential equations, and affine geometry \cite{nagy2019halfspace}. As robustness is a critical concern in data science applications, we posit that a deeper understanding of these depth functions and algorithms to compute them will also provide an avenue for the application of these ideas to more contemporary data problems.

Given a discrete measure described by a finite sample of points $\{x_1 \ldots x_d\}$, $x_i \in \R^n$, computation of the Tukey depth is a challenging problem. This is essentially related to the fact that any exact solution method must execute a combinatorial search over many possible hyperplanes defined by the $x_i$'s (a summary of some relevant literature is given in Section \ref{sec:lit}). The discrete nature of the points makes it difficult to extrapolate an optimal hyperplane from one point to other nearby points. Various heuristics have been proposed to improve computational efficiency of Tukey depths for finite point clouds, but the methods and analysis are often challenging.

In this paper, we study the population level problem, namely we consider the Tukey depth for an absolutely continuous measure $\mu$ (with density $\rho$). The Tukey depth of such a measure will satisfy necessary conditions which take the form of differential equations. In particular, we demonstrate that the Tukey depth formally satisfies the partial differential equation
\begin{equation}
  |\nabla T(x)| = \int_{(y-x) \cdot \frac{\nabla T}{|\nabla T|} = 0} \rho(y) d \mathcal{H}^{n-1}(y),
\label{eqn:IntroTukeyPDE}
\end{equation}
where $T$ is a function $T:\mathbb{R}^n \rightarrow \mathbb{R}$ and $d \mathcal{H}^{n-1}$ is the $n-1$-th dimensional Hausdorff measure.
This differential equation offers a potential avenue for leveraging local information to approximate the Tukey depth. However, the proper interpretation of this equation, and indeed whether it characterizes the Tukey depth, is not a simple matter. This work offers a first step in this direction, by linking the study of Tukey depths with the theory of viscosity solutions of partial differential equations. This formulation provide a new means of approximating Tukey depths. We provide a few numerical illustrations of this technique and its promise, but we set aside the important question of how to extend this formulation to empirical measures for future work.

On an informal level our contributions are as follows:
\begin{enumerate}
\item We identify a first-order partial differential equation \eqref{eqn:IntroTukeyPDE} that is satisfied by the Tukey depth at points where it is smooth. This equation is similar to the Eikonal equation from geometry, but possesses a ``right-hand side'' that does not fit with the classical theory.
\item We propose using the classical framework of viscosity solutions for Hamilton-Jacobi equations, which includes Eikonal equations, in order to recover well-posedness of those differential equations, and in order to characterize $T$ at any points where the Tukey depth is not differentiable. To remain self-contained, we do not assume the reader that is familiar with the theory of Hamilton-Jacobi equations. As the particular form of our differential equation is not covered by the standard theory, we extend the classical theory as needed, providing complete proofs.
\item We demonstrate that the proposed viscosity solutions always provide a lower bound for the Tukey depth. We provide non-trivial examples where the Tukey depth is and is not a viscosity solution to the differential equation, and then we provide one sufficient condition, for uniform distributions on convex domains in two dimensions, where Tukey depths are characterized as viscosity solutions. We conjecture a wide class of cases where the two may be equal.
\end{enumerate}
The organization of the work is as follows. In Section \ref{sec:lit} we review the relevant literature. In Section \ref{sec:Derivation} we offer the formal derivation of the differential equation. In Section \ref{sec:visc-sol} we identify why one needs a selection criteria to use this differential equation, and then develop the viscosity solution theory for this equation. In Section \ref{sec:depths-as-visc} we then prove that the viscosity solution provides a lower bound, and characterize one sufficient condition where the viscosity solution and the Tukey depth match. In Section \ref{sec:illustrations} we provide some basic numerical illustrations of using this formulation to approximate Tukey depths in cases where analytic solutions are known.

\subsection{Related Literature}\label{sec:lit}
Tukey depths were introduced in a number of early works \cite{Tukey-1975,hodges1955bivariate,Hotelling}.  Several basic properties of depth functions in the population setting, such as affine invariance, convexity of level sets, and some basic properties of minimizing hyperplanes at maxima, are given in \cite{rousseeuw1999depth}. Subsequently, these depth functions received attention due to their robustness properties. In particular, statistics based upon Tukey depths enjoy a high breakdown point \cite{donoho1992breakdown}, meaning that a large proportion of the data must be moved in order to appreciably affect the depth function. Later extensions of these concepts and properties to more general depth functions can be found in \cite{zuo2000general,zuo2003projection}. Exact solutions for several canonical examples such as uniform distributions on regular polygons and independent Gaussians and Cauchy distributions, can be found in \cite{rousseeuw1999depth}. Asymptotic consistency of depth functions in the empirical to population limits was studied in a number of works \cite{zuo2004stahel,ma2011asymptotic,nolan1992asymptotics,masse2004asymptotics}. Finally, there has been significant recent interest in generalized depth functions based upon Wasserstein distances \cite{chernozhukov2017monge}.

The computational geometry community has invested a significant effort in computational methods for the halfspace depth. Most exact computational formulations for the halfspace depth require evaluating empirical measures over many possible halfspaces. This can be cast as a combinatorial optimization problem, which generally does not scale well in dimension. For example, exact computation of the boundary of a level set of the Tukey depth of $n$ points in dimension $d$ has (neglecting log terms) $O(n^{d})$ time complexity \cite{liu2019fast,paindaveine2012computing}; this is one of the reasons that these methods are typically restricted to settings where $d$ is very small, i.e. 2 or 3. The computational complexity of methods for computing an exact solution to this problem are described in \cite{Chen-thesis}. Various heuristic or approximate approaches have  been proposed, including random projection methods \cite{cuesta2008random}, random sampling methods \cite{chen2013absolute} and geometric approximation algorithms \cite{bogicevic2018approximate}. One classical reference \cite{matousek1991computing} describes a $O(n)$ approximation algorithms for fixed error $\delta$, but is quick to point out that the big Oh notation obscures the dependence on $\delta$ in ways that are crucially important in practical computations. One such example, is \cite{chen2013absolute}, which lists a $O(\delta^{1-d} n)$ run time.  Although these approximation algorithms and heuristics offer some avenues of attack, computation of the Tukey depth is still a challenging problem.

In this paper, we consider an alternative approach. In particular, we notice that at the population level the necessary conditions take the form of a first-order partial differential equation, similar in form to the classical Eikonal equation. Equations of this type have been studied extensively in the context of optimal control; monographs introducing the topic include \cite{bardi2008optimal,katzourakis2014introduction}. Some works tackle the problem of HJB equations with right-hand sides that are not smooth \cite{elliott2004uniqueness,ishii1985hamilton}, but those works do not apply in our situation where the dependence is non-smooth in the gradient of the solution. A variety of efficient computational methods for viscosity solutions are also available. In particular, level set and fast marching methods \cite{sethian1999level,sethian1999fast} provide elegant, efficient means for approximating viscosity solutions. These methods have been extended to unstructured grids and graphs \cite{sethian2000fast,desquesnes2010efficient,desquesnes2013eikonal}, which suggests possible application to graphs based upon empirical distributions.

One line of work that is similar in spirit to ours is the work in \cite{calder2014hamilton,calder2015pde,calder2020limit,cook2020rates}. These works utilize Hamilton-Jacobi equations to study a variety of measures of depths among data points. For example, \cite{calder2020limit} studies the convex hull depth, an alternative to the Tukey depth, and characterizes a Hamilton-Jacobi equation that is satisfied in a continuum limit. Earlier works on Pareto fronts \cite{calder2014hamilton,calder2015pde} also identify appropriate Hamilton Jacobi equations to numerically approximate these fronts, which can be thought of as a type of ``depth'' for multi-objective optimization. Recent work \cite{cook2020rates} also quantifies the statistical consistency of such algorithms. These works complement ours, in that they, in the context of geometrical learning problems distinct from ours, demonstrate the promise of utilizing limiting Hamilton-Jacobi equations to understand properties of difficult geometric and combinatorial optimization problems.

Finally, there has also been a parallel literature within the convex geometry community studying \emph{convex floating bodies}, which is simply another name for the Tukey depth. These floating bodies were used to define a generalized notion of affine surface area, namely a notion of surface area that is invariant under volume preserving affine transformations \cite{schutt1990convex}. These affine surface areas enjoy a number of applications in sampling theory, differential geometry and information theory. An in depth review on the topic, with comparisons to the theory of Tukey depths, is given in \cite{nagy2019halfspace}.

\section{Formal derivation of partial differential equation} \label{sec:Derivation}

We consider a probability distribution on $\mathbb{R}^n$ with an absolutely continuous probability density $\mu$, that is, for any Borel set $E \subset \mathbb{R}^n$ we have $\mu(E)= \int_E \rho(x) dx$ for some integrable (in the sense of Lebesgue) function $\rho$. Additionally, we require that the support of $\rho$ is on $\overline{\Omega}$ for some open and bounded set $\Omega \in \mathbb{R}^n$. For concreteness, we restrict our attention to the bounded support case, but more general cases can be treated with appropriate modifications of the arguments in this work.
	
	We define the Tukey depth (also referred to in the literature as the halfspace depth) of a point $x \in \mathbb{R}^n$ in the following manner:
	\begin{equation}
	T(x) := \inf\{ \mu(H) : H \text{ is a closed halfspace containing $x$}\}.
	\label{eqn:Tukey-definition}
	\end{equation}
	
	We now formally derive a partial differential equation satisfied by the Tukey depth $T$, which can be roughly interpreted as a necessary condition for the Tukey depth minimization problem. This equation will be satisfied at any point that the Tukey depth is smooth. Furthermore, the viscosity solution formulation of these equations will turn out to also characterize the behavior of $T$ near singular points: further detail on the definition of viscosity solutions and their application to Tukey depths will be given in subsequent sections.
	
	Let us define the following function for $x \in \R^n$ and $\nu \in \mathbb{R}^n$:
	\begin{equation}
	Z(x,\nu) = \mu(\{ y: y \cdot \nu \geq x \cdot \nu\}).
	\label{eqn:massHalfspace}
	\end{equation}
	
	This function indicates the amount of mass of $\mu$ that lies on one side of the hyperplane that contains the point $x$ and is perpendicular to the direction $\nu$. We notice that the Tukey depth, according to Definition \eqref{eqn:Tukey-definition}, can be written in terms of the function $Z$ as
	\begin{equation}
	\label{eqn:Tukey-def-min}
	T(x) = \inf_{\nu \in \mathbb{S}^{n-1}} Z(x,\nu),
	\end{equation}
	where $S^{n-1}$ is the hypersphere in dimension $n-1$. In other words, we are restricting the minimization problem to the set of vectors $\nu \in \mathbb{R}^n$ with $|\nu|=1$.
	
	Let us assume that the mapping $\argmin Z(x,\cdot)$ is single valued and smooth in some neighborhood, and let us denote its value by $\nu(x)$, i.e., $T(x)=Z(x,\nu(x))$ locally. Then, assuming that $Z$ is smooth in both variables, by the chain rule we compute the total derivative of $T$:
	\[
	\nabla T = D_x Z + D_{\nu} Z \cdot D \nu .
	\]
	
	We can see $T$ as the minimization of $Z(x,\cdot)$ under the constraint $|\nu|=1$. According to the method of Langrange multipliers, at a minimizer we must have $D_{\nu} Z = k(x) \nu$ for some $k:\mathbb{R}^n \rightarrow \mathbb{R}$. We then have that
	\begin{equation}
	\label{eqn:gradT}
	\nabla T = D_x Z + k \nu \cdot D \nu = D_x Z + \frac{1}{2} k \, D (|\nu|^2) = D_x Z,
	\end{equation}
	where the second term in the equation vanishes because of the restriction $|\nu|=1$. 
	
	Let us note that the derivative of $Z(x,\nu)$ in the direction of a unit vector $h$ can be calculated directly from the definition:
	\begin{equation}
	\begin{array}{ccl}
	D_h Z (x,\nu) & =  & \displaystyle \lim_{t \rightarrow 0} \frac{Z(x+th,\nu)-Z(x,\nu)}{t}  =   - \sgn (h \cdot \nu) \lim_{t \rightarrow 0} \displaystyle \frac{1}{t} \int_{(x+th) \cdot \nu \geq y \cdot \nu \geq x \cdot \nu} \rho(y) dy \\
	& = & - \sgn (h \cdot \nu) \displaystyle \lim_{t \rightarrow 0} h \cdot \nu  \frac{1}{t} \int_{0}^{t} \int_{(y-x- \tau  h) \cdot \nu = 0} \rho(y) d \mathcal{H}^{n-1} (y) d \tau \\ & = &  - \displaystyle h \cdot \nu \int_{(y-x) \cdot \nu =0} \rho(y) d \mathcal{H}^{n-1}(y) \mbox{,}
	\end{array} 
	\label{eqn:DirDerZ}
	\end{equation}
	For the derivation of this expression we used the coarea formula for the change of variables $\displaystyle f(y)=\frac{(y-x) \cdot \nu}{h \cdot \nu}$, whose level sets $f(y)=c$ are planes with normal $\nu$ passing through the point $x+ch$.
	Under the assumption that $Z$ is differentiable, we then have the following expression for the differential of $Z$
	\begin{equation}
	D_x Z(x,\nu)= - \displaystyle \nu \int_{(y-x) \cdot \nu =0} \rho(y) \,d\mathcal{H}^{n-1}(y) \mbox{.}
	\end{equation}
	In view of Equation  \eqref{eqn:gradT} we have that
	\begin{equation}
	\nabla T = - \displaystyle \nu \int_{(y-x) \cdot \nu =0} \rho(y) \,d\mathcal{H}^{n-1}(y) \mbox{,}
	\label{eqn:Tukey-grad}
	\end{equation}
	from which it follows that $\nu=-\frac{\nabla T}{|\nabla T|}$. After multiplying Equation \eqref{eqn:Tukey-grad} by $\frac{\nabla T}{|\nabla T|}$ we obtain that
	\begin{equation}
	|\nabla T|  = \displaystyle \frac{\nabla T}{|\nabla T|} \cdot \frac{\nabla T}{|\nabla T|} \int_{(y-x) \cdot \frac{\nabla T}{|\nabla T|} =0} \rho(y) \,d\mathcal{H}^{n-1}(y) = \int_{(y-x) \cdot \frac{\nabla T}{|\nabla T|} =0} \rho(y) \,d\mathcal{H}^{n-1}(y) .
	\end{equation}
	
	Hence we find that the equation that the Tukey depth $T$ formally satisfies the differential equation
	\begin{equation}
	\label{eqn:TukeyEqn}
	|\nabla u|  =  \int_{(y-x) \cdot \frac{\nabla u}{|\nabla u|} =0} \rho(y) \,d\mathcal{H}^{n-1}(y),
	\end{equation}
	where we have denoted the unknown function by $u$ as it is customary in the theory of partial differential equations. This notation also emphasizes that we still need to determine under which assumptions and in what sense the Tukey depth is a solution of \eqref{eqn:TukeyEqn}.

	We mention that there are other ways to derive this type of differential equation. In particular, one can use the method of characteristics to derive a system of ordinary differential equations which is equivalent to \eqref{eqn:TukeyEqn} in regions where the solution is smooth. However, resolving the singularities is often very challenging within the context of the characteristic equations \cite{albano2002propagation}, and hence we do not pursue that formalism here.
	
	\begin{example} \label{ex:1D-Eikonal}
	  Consider the case where $d=1$, and $\rho(x) \equiv 1$ on the interval $[0,1]$ (namely the uniform distribution on $[0,1]$. In this case the Tukey depth is given by
	    \begin{displaymath}
	      T(x) = \begin{cases} \min(x,1-x) &\text{ for } x \in [0,1] \\ 0 &\text{otherwise.} \end{cases}
	    \end{displaymath}
	    Here Equation \eqref{eqn:TukeyEqn} reduces to $|u_x| = \rho(x) = 1$. One immediately verifies that $T$ satisfies this equation at every point where it is differentiable.
	  \end{example}
	  Indeed, if we consider a probability distribution in one dimension which has density $\rho$, supported on $[a,b]$, and so that $\rho$ is continuous on $(a,b)$, then the Tukey depth will solve the boundary value problem
\begin{equation}
	\label{eqn:Tukey-1D}
	|u_x|-\rho(x) = 0, \qquad u(a) = u(b) = 0
	\end{equation}
	at every point except for on the boundary of the set of medians.	

	\section{Definitions and properties of viscosity solutions}\label{sec:visc-sol}
	
	In Example \ref{ex:1D-Eikonal}, we saw that the Tukey depth solved the PDE \eqref{eqn:TukeyEqn} at almost every point. In many cases solving a PDE almost everywhere is enough to characterize the solution (for example in the context of weak solutions to elliptic PDE). However, in the context of first-order equations this is not the case, as evidenced by the following example.
	\begin{example} \label{ex:1D-inifinite-sol}
		Consider again the uniform distribution on $[0,1]$. Then there are multiple functions which satisfy \eqref{eqn:Tukey-1D} almost everywhere on $[0,1]$. For example, both $T$ and $-T$ satisfy the equation at any point where they are differentiable. Indeed, one can actually construct an infinite number of functions that satisfy the equation almost everywhere by considering functions of the form
		
		\begin{equation}
		u_m(x) = \left \{
		\begin{array}{lcl}
		\displaystyle x - \frac{k-1}{m} & \mbox{ if } & x \in \displaystyle \left[ \frac{k-1}{m}, \frac{k-1}{m} + \frac{1}{2m}\right] \\
		\displaystyle \frac{k}{m} - x & \mbox{ in } & \displaystyle x \in \left[ \frac{k}{m}-\frac{1}{2m}, \frac{k}{m} \right] \qquad \qquad k=1, \ldots m
		\end{array}
		\right.
		\label{eqn:InftySols1D}
		\end{equation}
	\end{example}
	\begin{figure}[h]
		\centering
		\includegraphics[width=0.75\textwidth]{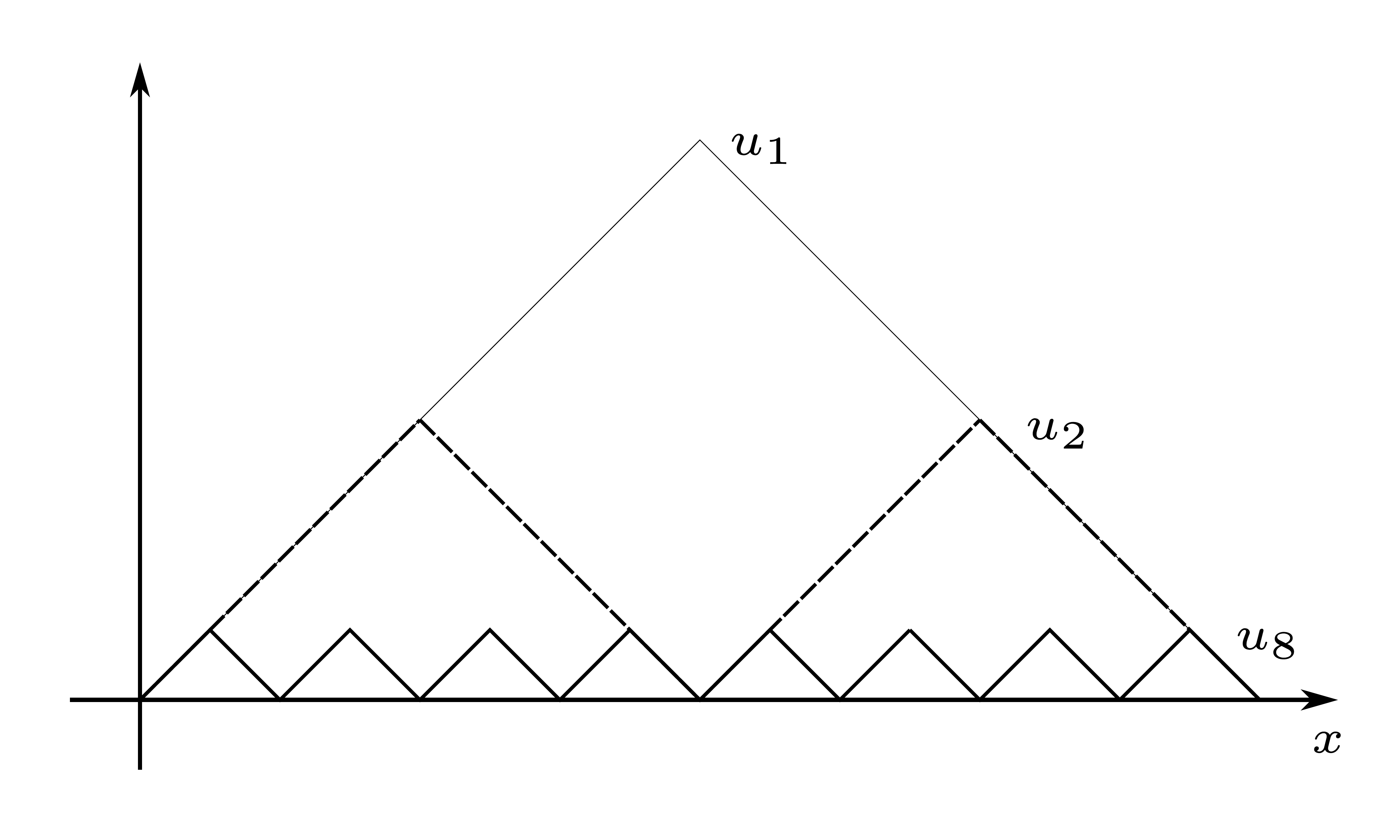}
		\caption{Shown here are some of the infinitely many functions given by formula \eqref{eqn:InftySols1D}. They all satisfy equation \eqref{eqn:Tukey-1D} almost everywhere on $[0,1]$ for the uniform distribution.}
	\end{figure}
	
	The previous example demonstrates that the ``almost everywhere'' solution formulation (or weak solution formulation) is not so useful in the context of these types of first-order PDE. This led to the (well-developed) theory of viscosity solutions, which we describe here to keep the exposition self-contained.
	
	\subsection{Classical theory of viscosity solutions}
	In order to rule out this undesirable situation when defining what we should understand by a solution of \eqref{eqn:Tukey-1D}, we draw upon the concept of viscosity solutions \cite{crandall1983viscosity}. Building upon Example \ref{ex:1D-Eikonal}, we define
	\begin{equation}
	H(x,p)= |p| - 1 ,
	\end{equation}
	for $(x,p) \in \mathbb{R}^2$.
	
	We can rewrite equation \eqref{eqn:Tukey-1D} as $H(x,D u) = 0$. We then need to define what this means at points where $u$ is not differentiable, and do so in a way that selects the ``correct'' solution (i.e. the Tukey depth) and precludes the spurious solutions from Example \ref{ex:1D-inifinite-sol}.
	
	We notice that the Tukey depth in this example has no ``downward pointing corners'', and the upward pointing corner is found by taking a minimum of two classical solutions to the differential equation. Indeed, the Tukey depth is the only weak (i.e. almost everywhere solution) of the differential equation with those properties that satisfies the $u(0) = u(1) = 0$. We thus only need to create a notion of generalized solution which selects these corners in an appropriate way. To do so, we recall the definition of the super- and subdifferential:
	
	\begin{equation}
	\begin{array}{ccl}
	  D^+ u (x) & = & \left\{p \in \mathbb{R}^n \left| \displaystyle \lim_{y \rightarrow x } \sup_{y \in \Omega} \frac{u(y)-u(x)-p \cdot (y-x)}{|x-y|} \leq 0 \right. \right\} \\ \\
	  D^- u (x) & = & \left\{p \in \mathbb{R}^n \left| \displaystyle \lim_{y \rightarrow x } \inf_{y \in \Omega} \frac{u(y)-u(x)-p \cdot (y-x)}{|x-y|} \geq 0 \right. \right\} \\
	\end{array}
	\label{eqn:semidiffs}
	\end{equation}
	
	If a vector $p$ is in $D^+u$ ($D^-u$) then we can understand it as the slope of a plane (or line) which locally supports the function from above (below). Furthermore, if for some $x$, $D^+(x)$ and $D^-(x)$ are nonempty then $u$ is differentiable at $x$ \cite{bardi2008optimal}.
	
		In order to formalize the concept of viscosity solutions, we recall the definitions of lower and upper semicontinuous functions:
		
		\begin{definition}
			A function $u:\mathcal{O} \subset \mathbb{R}^K \rightarrow \mathbb{R}$ is called upper semicontinuous at $x \in \mathcal{O}$ if
			\begin{equation}
			\inf_{r > 0} \,\, \sup_{y \in \mathcal{O} \cup B(x,r)} u(y)=:\limsup_{y \rightarrow x} u(y) \leq u(x).
			\end{equation}
			Similarly, $u:\mathcal{O} \subset \mathbb{R}^K \rightarrow \mathbb{R}$ is called lower semicontinuous at $x \in \mathcal{O}$ if
			\begin{equation}
			u(x) \leq \liminf_{y \rightarrow x} u(y):= \sup_{r > 0} \,\, \inf_{y \in \mathcal{O} \cup B(x,r)} u(y).
			\end{equation}	
		\end{definition}
		
		We will also need to introduce the concepts of lower and upper semi continuous envelopes.
		
		\begin{definition}
			The upper semicontinuous envelope $u^*$ of a function $u:\mathcal{O} \subset \mathbb{R}^K \rightarrow \mathbb{R}$ is defined as
			\begin{equation}
			u^*(x)=\limsup_{y \rightarrow x} u(y),
			\end{equation}
			and the lower semicontinuous envelope $u_*$ of a function $u:\mathcal{O} \subset \mathbb{R}^K \rightarrow \mathbb{R}$ is defined as
			\begin{equation}
			u_*(x)= \liminf_{y \rightarrow x} u(y)
			\end{equation}
		\end{definition}
		
		For any function $u$, we have from the definitions that $u^*$ and $u_*$ are respectively upper and lower semicontinuous and that $u_* \leq u \leq u^*$. It is also immediate that $u^*$ is the smallest upper semicontinuous function larger than or equal to $u$, and that $u_*$ is the largest lower semicontinuous function smaller than or equal to $u$. 
		
	Having defined the super and sub-differentials and the semicontinuous envelopes of a function, we may now state a definition for viscosity solutions. This definition is slightly more general than those given in classical references \cite{bardi2008optimal}, as we need to apply the theory for $H$ that is not continuous. References which use this definition include Katzourakis \cite{katzourakis2014introduction} and Ishii \cite{ishii1985hamilton}, but those references do not directly handle the types of discontinuities that we have in our context, and hence we present the necessary adaptation of the theory in detail. Our presentation follows that of \cite{calder2018lecture} and \cite{bardi2008optimal} for continuous $H$.

	\begin{definition}
		\label{def:viscositySol}
		An upper semicontinuous function $u$ is a subsolution to $H(x,u,Du)=0$ if
		\begin{equation} \label{eqn:sub-sol-ineq}
		H_*(x,p) \leq 0 \mbox{ for all } x\in \Omega \mbox{ and all } p \in D^+u(x) \mbox{,}
		\end{equation}
		where $H_*$ is the lower semicontinuous envelope of $H$. Similarly, we call a lower semicontinuous function $u$ a supersolution to $H(x,u,Du)=0$ if 
		\begin{equation}
		H^*(x,p) \geq 0 \mbox{ for all } x\in \Omega \mbox{ and all } p \in D^-u(x) \mbox{,}
		\end{equation}
		
		where $H^*$ is the upper-semicontinuous envelope of $H$. Finally, a continuous function $u$ is called a \emph{viscosity solution} if it is both a subsolution and a supersolution.
	\end{definition}
	In the context of Example \ref{ex:1D-Eikonal}, the supersolution concept precludes any downward pointing corners, as at a downward corner $0 \in D^- u$ and $H(x,0) =-1< 0$. The supersolution concept also requires that $|u_x| \geq 1$ anywhere that it exists. Similarly, the subsolution concept enforces that $|u_x| \leq 1$ anywhere that it exists. These points together turn out to guarantee that the only function which is a viscosity solution (i.e. which is both a super and subsolution) is given by the Tukey depth (in one dimension).
	
	The notion of viscosity solutions is notably dependent the sign chosen: if we replaced $H(x,p) = 0$ with $-H(x,p) = 0$ we would obtain a completely different viscosity solution; in the context of Example \ref{ex:1D-Eikonal} we would obtain $-T$. We also mention that, in cases where viscosity solutions exist they may be obtained as a limit of ``regularized'' PDES, namely as the limit of solutions to the elliptic equations $H(x,\nabla u) = \varepsilon \Delta u$ as $\varepsilon \to 0$: this is the origin of the name ``viscosity solution'', and offers many theoretical and numerical tools for finding solutions of the same \cite{crandall1983viscosity}. 
	
	Finally, there is any equivalent formulation of viscosity solutions, which is more convenient for certain proofs and computations, and which we give here:
	
	\begin{definition}
		\label{def:viscositySolEquiv}
		We say that an upper semicontinuous (lower semicontinuous ) function $u$ is a subsolution (supersolution) of the equation $H(x,D u)=0$ if 
		\begin{equation}
		\begin{array}{c}
		H_* (x, D \varphi (x)) \leq 0 \\
		(H^*(x, D \varphi (x)) \geq 0)
		\end{array}
		\end{equation}
		for any $\varphi \in \mathcal{C}^1(\Omega)$ such that $u - \varphi$ attains a local maximum (local minimum) at $x$. 
	\end{definition}
	More explicitly, in the context of Equation \eqref{eqn:TukeyEqn} we have that
	\begin{equation}
	\begin{array}{ccl}
	H_*(x, D \varphi(x)) & := & \displaystyle \liminf_{q \rightarrow D \varphi (x)}  \left[ |q| - \int_{q \cdot (\xi - x) =0} \rho (\xi) \,d\mathcal{H}^{n-1} \right] \\
	& = & \displaystyle |D \varphi(x)| - \limsup_{q \rightarrow D \varphi(x)} \int_{q \cdot (\xi - x) =0} \rho (\xi) \,d\mathcal{H}^{n-1}
	\end{array}
	\end{equation}
	and
	\[
	H^*(x,D \varphi(x)) = |D \varphi(x)| - \liminf_{q \rightarrow D \varphi(x)} \int_{q \cdot (\xi - x) =0} \rho (\xi) \,d\mathcal{H}^{n-1} .
	\]
	
	\subsection{Properties of viscosity solutions: well-posedness and maximum principle}
	\label{sec:maxprinciple}
	
	One of the main advantages of the viscosity solution approach is that it recovers the existence and uniqueness that is lost when we only consider weak solutions to the first-order PDE. This well-posedness is a consequence of the \emph{comparison principle}, which is known to hold for the viscous approximations (i.e. when $\varepsilon \Delta u$ is added to the equation), and which is inherited by the viscosity solutions. In this section we will establish the maximum principle within the context of \eqref{eqn:TukeyEqn}, following the classical comparison principle for Hamilton-Jacobi equations and modifying the proofs as needed. The consequence of that maximum principle will be the following well-posedness result for Equation \eqref{eqn:TukeyEqn}:
	
	\begin{theorem}
		\label{th:exuniq}
		Consider the equation
		\begin{equation}
		H(x,D u) = |D u|  -  \int_{(y-x) \cdot D u =0} \rho(y) \,d\mathcal{H}^{n-1}(y) =0,
		\label{eqn:HequationTukey}
		\end{equation}
		where $\rho$ is a uniformly continuous function with support $\overline{S}$ for an open and bounded set $S \subset \mathbb{R}^n$, where $\Omega= co(S)$, the convex hull of $S$. Then there is a unique viscosity solution $u$ of this problem that satisfies the boundary condition $u=0$ on $\partial \Omega$, if we assume that
		\begin{equation}
		\label{eq:sliceassumption}
		\int_{p \cdot (\xi - x) = 0} \rho (\xi) \,d\mathcal{H}^{n-1} \geq \delta(x) \mbox{\qquad for all \quad}  x \in \Omega
		\end{equation}
		for any $p \in \mathbb{R}^n$, with $\delta(x)=0$ only if $x \in \partial \Omega$.
	\end{theorem}
	
	To prove this theorem, we must first establish a comparison principle for viscosity solutions of equation \eqref{eqn:HequationTukey}.
	\begin{theorem}
		\label{th:ComparisonPrinciple}
		Let $u$ and $v$ be respectively a subsolution and a supersolution of \eqref{eqn:HequationTukey}. If $u \leq v$ on $\partial \Omega$ then $u \leq v$ in $\Omega$.
	\end{theorem}
	
	To prove this fundamental comparison principle, we will construct an approximating sequence $u_k \to u$ so that $u_k$ are \emph{strict subsolutions}, meaning that the inequality \eqref{eqn:sub-sol-ineq} is a strict inequality. It turns out that the comparison principle is easier to prove for strict subsolutions. We begin with the following lemma:
		
	\begin{lemma}
		Let $u$ be a subsolution of \eqref{eqn:HequationTukey}.  Then the function $u_k = (1 - \varepsilon_k) u$, with $0<\varepsilon_k<1$, is a strict subsolution of \eqref{eqn:HequationTukey}, meaning that Equation \eqref{eqn:sub-sol-ineq} is a strict inequality for all $x \in \Omega$.
	\end{lemma}
	
	\begin{proof}
		
		According to Definition \ref{def:viscositySol}, in order to show that $u$ is a solution to \eqref{eqn:HequationTukey} we need to check that $H_*(x,p) \leq 0$ is satisfied for all $p \in \mathbb{R}^n$ in the superdifferential of $u$ at $x$.
		
		Let $p_k$ be in the superdifferential of $u_k$ at the point $x$. By the definition of $u_k$, we have that $p_k=(1-\varepsilon_k) p$ for some $p$ in the subdifferential of $u$. We need to consider two cases.
		
		If $|p| \leq \displaystyle \frac{\delta(x)}{2}$, then
		\begin{equation}
		\begin{array}{ccl}
		H_*(x,p_k) & = & |p_k|- \limsup_{q \rightarrow p_k} \displaystyle \int_{q \cdot (\xi - x) =0} \rho(\xi) \,d\mathcal{H}^{n-1} \\ & = & |p| - \varepsilon_k |p|- \limsup_{q \rightarrow p} \displaystyle \int_{q \cdot (\xi - x) =0} \rho(\xi) \,d\mathcal{H}^{n-1} \\
		& \leq & \displaystyle \frac{\delta(x)}{2} (1 - \varepsilon_k) - \delta(x) = -\displaystyle \frac{\delta(x)}{2} (1+ \varepsilon_k) \mbox{.}
		\end{array}
		\end{equation}
		
		If $|p| > \displaystyle \frac{\delta(x)}{2}$, then
		\begin{equation}
		\begin{array}{ccl}
		H_*(x,p_k) & = & |p| - \varepsilon_k |p| - \limsup_{q \rightarrow p_k} \displaystyle \int_{q \cdot (\xi - x)=0} \rho(\xi) \,d\mathcal{H}^{n-1} \\
		& = & |p| - \varepsilon_k |p| - \limsup_{q \rightarrow p} \displaystyle \int_{q \cdot (\xi - x) =0} \rho(\xi) \,d\mathcal{H}^{n-1} \\
		& \leq & - \varepsilon_k |p| < - \frac{\delta(x)}{2} \varepsilon_k
		\end{array}
		\end{equation}
		Thus we have that $H_*(x,p_k) \leq -\frac{\delta(x)}{2} \varepsilon_k$ and therefore $u_k$ is a strict subsolution of \eqref{eqn:TukeyEqn}.
	\end{proof}
	
	We now prove our comparison principle. The proof follows the classical ``doubling variables argument'', which doubles the number of variables and adds a quadratic regularization in order to construct appropriate functions that touch from above and below. We extend the proof in order to address the lack of continuity of $H$ and the possibility that densities go to zero near the boundary of their support. We remark that the proof likely can be adapted to handle other classes of distributions (e.g. densities with unbounded support and appropriate decay).
	
	\begin{proof}[Proof of Theorem \ref{th:ComparisonPrinciple}]
		
		Consider the function $u_k=(1- \varepsilon_k) u$ with $u_k \leq v$ on $\partial \Omega$, with $0<\varepsilon_k<1$. We want to show that $u_k \leq v$ on $\Omega$. The proof is carried out by contradiction. Assume that $\sup_{\overline{\Omega}} (u_k - v) >0$. Define the following function on $\overline{\Omega} \times \overline{\Omega}$:
		\begin{equation}
		\Phi(x,y) = u_k(x) - v (y) - \frac{\alpha}{2} |x-y|^2 
		\end{equation}
		The function is upper semicontinuous and, because $\Omega$ is bounded, it attains its maximum at some point $(x_{\alpha},y_{\alpha}) \in \overline{\Omega} \times \overline{\Omega}$.
		
		As $u,-v$ are bounded above on $\overline{\Omega}$ then $|x_{\alpha} - y_{\alpha}|^2 \leq \frac{C}{\alpha}$. We can assume that, as $\alpha \rightarrow \infty$, there is a subsequence $\{(x_{\alpha}, y_{\alpha})\} \rightarrow (x_0,x_0)$ for some $x_0 \in \overline{\Omega}$. 

		By the assumptions on the boundary values we must have that $x_0 \in \Omega$. Therefore $(x_{\alpha},y_{\alpha}) \in \Omega \times \Omega$ for large $\alpha$.
		
		Consider the function $\varphi(x) = \frac{\alpha}{2} |x - y_{\alpha}|^2$. By the definition of $\Phi$, $u - \varphi$ has a local maximum at $x_{\alpha} \in \Omega$. Then
		\begin{equation}
		H_* (x_{\alpha}, \alpha (x_{\alpha} - y_{\alpha})) \leq - \frac{\delta(x_{\alpha})}{2} \varepsilon_k
		\label{eqn:HLowerComparison}
		\end{equation}
		
		In the same manner, if we define the function $\psi (y) = - \frac{\alpha}{2} |x_{\alpha}- y |^2$ then $v - \psi$ has a local minimum at $y_{\alpha}$ and therefore
		\begin{equation}
		H^*(y_{\alpha}, \alpha (x_{\alpha}-y_{\alpha})) \geq 0
		\label{eqn:HUpperComparison}
		\end{equation}
		Note that if $x_{\alpha}= y _{\alpha}$, then this remark implies that $0$ is in the subdifferential of $v$ at the point $y_{\alpha}$. This contradicts the assumption that $v$ is a supersolution, namely that
		\begin{equation}
		\limsup_{q \rightarrow 0} \left[ |q| - \int_{q \cdot (\xi-x) = 0 } \rho (\xi) \,d\mathcal{H}^{n-1} \right] = 0 - \liminf_{q \rightarrow 0} \int_{q \cdot (\xi -x ) =0} \rho(\xi) \,d\mathcal{H}^{n-1} \geq 0
		\end{equation} 
		in view of assumption (\ref{eq:sliceassumption}). We can therefore be certain that $x_{\alpha} \neq y_{\alpha}$.
		
		Combining inequalities \eqref{eqn:HLowerComparison} and \eqref{eqn:HUpperComparison} yields
		\begin{equation}
		\label{eq:upperlowerdiff}
		0 < \delta \varepsilon_k \leq \frac{\delta(x_{\alpha})}{2} \varepsilon_k  \leq H^*(y_{\alpha}, \alpha (x_{\alpha}-y_{\alpha})) - H_* (x_{\alpha}, \alpha (x_{\alpha} - y_{\alpha})) , 
		\end{equation}
		where we know that $\frac{\delta (x_{\alpha})}{2} > \delta$ for some $\delta > 0$ because $x_{\alpha}$ is well-separated from the boundary for large enough $\alpha$.
		
		Note that for $p \neq 0$ we have that
		\begin{equation}
		\begin{array}{ccl}
		|H^*(y,p) - H_* (x,p)| & = & |H(y,p)- H(x,p)| \\
		& = & \left| \displaystyle \int_{p \cdot (\xi - y) = 0} \rho (\xi) \,d\mathcal{H}^{n-1}(\xi) - \int_{p \cdot (\xi - x) = 0} \rho (\xi) \,d\mathcal{H}^{n-1} (\xi) \right| \\
		& = & \left| \displaystyle \int_{p	 \cdot (\zeta - x) = 0} \rho (\zeta - x + y) \,d\mathcal{H}^{n-1} (\zeta) - \int_{p \cdot (\xi - x) = 0} \rho (\xi) \,d\mathcal{H}^{n-1}(\xi) \right| \\
		& = & \left|\displaystyle \int_{p \cdot (\xi - x) = 0} \left[ \rho(\xi - x + y) - \rho (\xi) \right] \,d\mathcal{H}^{n-1} (\xi) \right| \\
		& \leq & \displaystyle \int_{p \cdot (\xi - x) =0} \mathbf{1}_{\Omega} \,\,\omega (|y -x|) \,d\mathcal{H}^{n-1}(\xi)  \\ &= & \omega (|y - x|) \displaystyle \int_{p \cdot (\xi - x) = 0} \mathbf{1}_{\Omega} \,\, \,d\mathcal{H}^{n-1} (\xi) \leq C \mbox{diam} (\Omega) ^{n - 1} \omega (|y - x|)
		\end{array}
		\end{equation}
		where we have used the change of variables $\zeta = \xi + x - y$ and $\omega$ stands for the modulus of continuity of the density $\rho$. 
		
		Thus the inequality (\ref{eq:upperlowerdiff}) implies that
		\begin{equation}
		\begin{array}{{ccl}}
		0 < \delta \varepsilon_k & \leq &  H^*(y_{\alpha}, \alpha (x_{\alpha}-y_{\alpha})) - H_* (x_{\alpha}, \alpha (x_{\alpha} - y_{\alpha})) \\
		& \leq & C \mbox{diam} (\Omega)^{n-1} \omega(|y_{\alpha} - x_{\alpha}|)
		\end{array} \mbox{,}
		\end{equation}
		which yields a contradiction when $\alpha \rightarrow \infty$ as both $\delta$ and $\varepsilon_k$ are fixed and strictly greater than zero. Thus we have that $u_k \leq v$ in $\Omega$. Finally, we only need to take a sequence $\{\varepsilon_k\}$ such that $\varepsilon_k \rightarrow 0$ when $k \rightarrow \infty$ to conclude that $u \leq v$.
	\end{proof}
	
	The comparison principle immediately implies the following uniqueness property of viscosity solutions.
	
	\begin{corollary}\label{cor:uniqueness}
		Let $u_1$ and $u_2$ be viscosity solutions of \eqref{eqn:TukeyEqn}, and suppose that $u_1 = u_2$ on $\partial \Omega$. Then $u_1 = u_2$.
	\end{corollary}
	
	Having established the comparison principle and uniqueness, we now turn to proving the desired existence of solutions. In order to prove Theorem \ref{th:exuniq}, we want to use the Perron method for viscosity solutions of equations of the type 
	\begin{equation}
		\label{eqn:HxDu}
		H(x,Du)=0 .
	\end{equation}
	Given a supersolution $w$ of equation \eqref{eqn:HxDu}
	let us define
	\begin{equation}
	\mathcal{F}:= \left\{ u \Big| u \mbox{ is a viscosity subsolution of \eqref{eqn:HxDu} and } u \leq w \mbox{ in } \overline{\Omega}   \right\} .
	\end{equation}
Define now the function
	\begin{equation}\label{eqn:sup-U-def}
	U(x) := \sup \left\{ u(x) | u \in \mathcal{F} \right\}.
	\end{equation}
	As long as the set $\mathcal{F}$ is non-empty, the function $U$ is finite and well-defined, and provides a candidate viscosity solution. Informally, the method of proving that $U$ is a viscosity solution, also called the \emph{Perron method} is to demonstrate that $U^*$ and $U_*$, namely the upper-semicontinuous and lower semicontinuous envelopes of $U$, are, respectively, subsolutions and supersolutions to \eqref{eqn:TukeyEqn}. The maximum principle, namely Theorem \ref{th:ComparisonPrinciple}, then implies that $U^* \leq U_*$, which implies that $U_*=U=U^*$ and that $U$ is a viscosity solution.
	
	In order to use the Perron method, we need two lemmas.
	\begin{lemma}
		\label{lemma:sub}
		If $\mathcal{F}$ is nonempty then $U^*$, with $U$ defined by \eqref{eqn:sup-U-def}, is a viscosity subsolution of \eqref{eqn:TukeyEqn}.
	\end{lemma}
	\begin{proof}
		Let $\varphi \in \mathcal{C}^1$ such that $U^*- \varphi$ has a local maximum at $x_0 \in \Omega$. We can assume without loss of generality that $\varphi(x_0)=U^*(x_0)$, by adding a constant to $\varphi$ if necessary. 	
		Thus we have that
		\begin{equation}
		U^*(x)-\varphi(x) \leq 0 \mbox{ on } B(x_0,r) ,
		\end{equation}
		for some $r>0$.
		Let us define the function
		\begin{equation}
		\widetilde{\varphi} (x) = \varphi(x)+|x-x_0|^2 
		\end{equation}
		We note that  $U^*-\widetilde{\varphi}$ has a strict local maximum at $x_0$, and in particular
		\begin{equation}
		\label{eqn:sublemmaineq}
		U^*(x)- \widetilde{\varphi}(x) \leq -|x-x_0|^2 \mbox{ on } B(x_0,r).
		\end{equation}
		
		By the definition of $U^*$ as the upper semicontinuous envelope of $U$, there exists a sequence $x_k \rightarrow x_0$ such that $U(x_k) \rightarrow U^*(x_0)$. As $U$ is pointwise defined as  the supremum of the elements in $\mathcal{F}$, for each positive integer $k$ we can find a $u_k \in \mathcal{F}$ such that 
		\begin{equation}
		u_k(x_k) \geq U(x_k) - \frac{1}{k}.
		\end{equation}
		
		Because each $u_k$ is a subsolution of \eqref{eqn:HxDu} and is therefore upper semicontinuous, the function $u_k-\varphi$ attains a maximum at some point $y_k \in B(x_0,r)$. In view of inequality \eqref{eqn:sublemmaineq}, we have that
		
		\begin{equation}
		\begin{array}{ccccl}
		|y_k-x_0|^2 & \leq & \widetilde{\varphi}(y_k) - U^*(y_k)& \leq & \widetilde{\varphi}(y_k) - u_k(y_k) \\
		& \leq & \widetilde{\varphi}(x_k) - u_k(x_k ) & \leq & \displaystyle \widetilde{\varphi}(x_k)-U(x_k) + \frac{1}{k},
		\end{array}
		\end{equation}
		where we have used that $U^* \geq U \geq u_k$ and the definitions of $x_k$ and $y_k$. Thus we have that
		\begin{equation}
		\lim_{k \rightarrow \infty} |y_k-x_0|^2 \leq \lim_{k \rightarrow \infty} \left(\widetilde{\varphi}(x_k) - U(x_k)+\frac{1}{k}\right)= \widetilde{\varphi}(x_0)-U^*(x_0)=0
		\end{equation}
		and therefore $\{y_k\}$ converges to $x_0$. This implies that for $k$ large enough $u_k - \widetilde{\varphi}$ attains a strict local maximum at $y_k$ in the interior of $B(x_0,r)$
		As $u_k$ is a subsolution of \eqref{eqn:HxDu} we have
			\begin{equation}
			H_*(y_k,D \widetilde{\varphi}(y_k)) \leq 0
			\end{equation}
		Finally, because $H_*$ is lower semicontinuous and the definition of $\widetilde{\varphi}$ implies that $$\lim_{y_k \rightarrow x_0} (y_k,D \widetilde{\varphi}(y_k))= (x_0, D\varphi(x_0)),$$ we have
			\begin{equation}
			H_*(x_0,D \varphi (x_0)) \leq \liminf_{k \rightarrow \infty} H_* (y_k, D \widetilde \varphi (y_k)) \leq 0,
			\end{equation}
		which shows that $U^*$ is a subsolution.
	\end{proof}
	
	The following result and the \textit{bump construction} \cite{crandall1992user} used in its proof are classical in the theory of viscosity solutions \cite{crandall1992user,crandall1983viscosity,calder2018lecture}. The proof we give mirrors the modern and clear presentation in \cite{calder2018lecture} and \cite{bardi2008optimal} with some modifications to take into account the more general definition of viscosity solution given in \eqref{def:viscositySol}.
	
	\begin{lemma}
		\label{lemma:super}
		Let $u \in \mathcal{F}$. If $u_*$ is not a viscosity supersolution of \eqref{eqn:TukeyEqn} then $v(x) > u(x)$ for some $v \in \mathcal{F}$ and some $x \in \Omega$. 
	\end{lemma}
	\begin{proof}
		The idea of the proof is to directly construct a subsolution $v$ in $\mathcal{F}$ with the required properties using $u$.
		Let $u \in \mathcal{F}$ and let us assume that $u_*$ is not a supersolution of \eqref{eqn:TukeyEqn}. Then there exists a function $\varphi \in \mathcal{C}^{\infty} (\Omega)$ such that $u_*- \varphi$ attains a local minimum at $x_0$ and
		\begin{equation}
		H^*(x_0,D \varphi(x_0))  <0 ,
		\label{eqn:maximcontr}
		\end{equation}
		We can assume that the local minimum at $x_0$ is equal to zero, i.e. $\varphi(x_0)=u_*(x_0)$, by adding a constant to $\varphi$ if necessary.
		
		By the definition of $\mathcal{F}$, we have that $u_* \leq w$ . This implies that $ 0 \leq u_* -\varphi \leq w-\varphi$ for $x \in B(x_0,r)$ for some $r>0$, because zero is the local minimum of $u_*-\varphi$ at $x_0$. We cannot have that $\varphi(x_0)=w(x_0)$, because this would imply that $w-\varphi$ has a local minimum at $x_0$ and contradict that $w$ is a supersolution in view of \eqref{eqn:maximcontr}. Then we have that $w(x_0)-\varphi(x_0)>0$. 
		
		Because $H^*$ is upper semicontinuous, there is $\varepsilon > 0 $ and a ball $B(x_0,r) \subset \Omega$ such that $\varphi \leq u_*$ and $\varphi + \varepsilon \leq w$ and
		\begin{equation}
		H^*((x,D \varphi(x)) +\varepsilon \leq 0
		\end{equation}
		for $x \in B(x_0,r)$. Let us define
		\begin{equation}
		\psi(x)=\varphi(x) + \delta \left(\left(\frac{r}{2} \right)^2 -|x-x_0|^2 \right).
		\end{equation}
		Again, by the upper semicontinuity of $H^*$, it is possible to choose a small enough $\delta >0$ so that $\psi \leq w$ on $B(x_0,r)$ and
		\begin{equation}
		H^*(x, D \psi(x)) \leq 0
		\label{eqn:psiissubsolution}
		\end{equation}
		on $B(x_0,r)$. Let us define the function $v$ as
		\begin{equation}
		v(x) = \left\{
		\begin{array}{ll}
		\max \{ u(x), \psi(x) \} & \mbox{ for } x \in B(x_0,r) \\
		u(x) & \mbox{anywhere else}
		\end{array}
		\right.
		\end{equation}
		
		Note that $\psi$ is a viscosity subsolution in view of equation \eqref{eqn:psiissubsolution} and therefore $v$ is a subsolution on $B(x_0,r)$. By the definition of $\psi$, we have $\psi(x) \leq \varphi(x)$ in $B(x_0,r) \backslash B(x_0,\frac{r}{2})$. On the other hand, we chose $\varepsilon$ so that $\varphi \leq u_* $ ($\leq u$) on $B(x_0,r)$, so we have that $v=u$ on $B(x_0,r) \backslash B(x_0,\frac{r}{2})$. Hence $v$ is a subsolution with $v \leq w$ and then $v \in \mathcal{F}$.
		
		By the definition of $u_*$, there is a sequence $x_k \rightarrow x_0$ such that $u(x_k) \rightarrow u_*(x_0)$. By the definition of $v$, we have $v(x) \geq \psi(x)$ on $B(x_0,r)$ and then:
		\begin{equation}
		\liminf_{k \rightarrow \infty} v(x_k) \geq \lim_{k \rightarrow \infty} \psi(x_k) = \lim_{k \rightarrow \infty} \varphi(x_k) + \delta \left(\frac{r}{2}\right)^2 = u_*(x_0)  + \delta \left(\frac{r}{2}\right)^2 .
		\end{equation}
		
		Thus, for $k$ large enough,
		we have that
		\begin{equation}
		v(x_k) \geq u(x_k) + \delta \left(\frac{r}{2}\right)^2
		\end{equation} 
		so that $v(x_k) > u(x_k)$, which concludes the proof.

	\end{proof}
	
	With these tools in hand, we can now prove the existence and uniqueness of viscosity solutions of equation \eqref{eqn:HequationTukey}

	\begin{proof}[Proof of Theorem \ref{th:exuniq}]
		We will prove in Section \ref{sec:Tukeyaresuper} that $T(x)$ is a viscosity supersolution of \eqref{eqn:HequationTukey} in $\Omega$. Let us take $w=T$ in the definition of $\mathcal{F}$, and $U$ as given in \eqref{eqn:sup-U-def}. We can be sure that $\mathcal{F}$ is not empty because $0$ is a subsolution of \eqref{eqn:HequationTukey}. As we have that $U \leq T$ in $\Omega$ and $T$ is continuous then we also have that $U^* \leq T$ in $\Omega$, by the characterization of the upper semicontinuous envelope of $U$ as the smallest upper semicontinuous function equal to or greater than $U$. In particular, $U^*(x) \leq T(x)=0$ for $x \in \partial \Omega$.
		
		Now, because $U^*$ is a subsolution of \eqref{eqn:HequationTukey}, according to Lemma \ref{lemma:sub}, and $U^* \leq T$, then $U^* \in \mathcal{F}$. By the definition of $U$, we then have $U \geq U^*$, and hence $U=U^*$. Additionally we have that $0 \leq U$ in $\overline{\Omega}$ and that, by the characterization of $U_*$ as the largest lower semicontinuous function greater than or equal to $U$, $0 \leq U_*$ in $\Omega$. In particular $0 \leq  U_*(x)$,  for $x \in \partial \Omega$.
		
		By applying the contrapositive of Lemma \ref{lemma:super} for $u = U$, using the definition for $U$ \eqref{eqn:sup-U-def} we have that $U_*$ is a viscosity supersolution of \eqref{eqn:TukeyEqn}. Because $U^* \leq 0 \leq U_*$ on $\partial \Omega$ we can use the comparison Theorem \ref{th:ComparisonPrinciple} to get that $U^* \leq U_*$ in $\Omega$.
		
		This then implies that $U=U^*=U_*$ in $\Omega$ and therefore $U$ is a viscosity solution to \eqref{eqn:TukeyEqn}. The uniqueness property in Corollary \ref{cor:uniqueness} then gives uniqueness and completes the proof.
	\end{proof}

	\section{Tukey depths as viscosity solutions}\label{sec:depths-as-visc}
	Motivated by the one-dimensional situation, we have developed a theory of viscosity solutions which applies to the first-order PDE that Tukey depths solve at points where they are smooth. However, we have not yet rigorously made any connection between the viscosity solutions of the PDE and the Tukey depths themselves. The following sections partially address this question. In particular, we demonstrate that Tukey depths are supersolutions and that, in certain cases with underlying convexity, Tukey depths are actually viscosity solutions. We also provide a specific example where the Tukey depth is not a viscosity solution. Fully resolving the question of when Tukey depths match the viscosity solution is a topic of current work. In any case, the viscosity solution provides a new avenue for computing lower bounds on Tukey depths in all situations, and for actually computing them in many situations.
	\subsection{Tukey depths are supersolutions}
	\label{sec:Tukeyaresuper}
	In view of Definition \ref{def:viscositySol}, we are interested in a precise characterization of the super- and subdifferentials, $D^+T(x)$ and $D^-T(x)$, of the Tukey depth. We shall see that the super- and subdifferentials of a function $u$ of the form $u(x) := \inf_{b \in B} g(x,b)$, in our case $g:S \times \mathbb{S}^{n-1} \rightarrow \mathbb{R}$ for some open $S \subset \mathbb{R}^n$, have certain special properties and characterization under differentiability assumptions on $g$. These functions are also called \emph{marginal} in the theory of optimal control.
	For the proofs of Propositions \ref{prop:supsub1} and \ref{prop:supsub2}, which are classical results in this area, we refer the reader to Chapter 2 of \cite{bardi2008optimal}.
	
Let us define the (possibly empty) set
\begin{equation}
M(x) := \{b \in B | u(x) = g(x,b)\},
\label{eqn:M-def}
\end{equation}
which contains the elements $b \in B$ that minimize $g$ for a given $x$. 

	\begin{proposition}
	  Assume that $g(x,b)$ is bounded for all $x \in S \subset \mathbb{R}^n$ and that $x \mapsto g(x,b)$ is continuous in $x$, uniformly with respect to $b \in B$, where $B \subset \mathbb{R}^n$ is a compact set, namely
		\begin{equation}
		\label{eq:gunifcont}
		|g(x,b)-g(y,b)| \leq \omega(|x-y|,R)
		\end{equation}
		for all $|x|,|y| \leq R, b \in B$ for some modulus of continuity $\omega$. Then for $u(x) := \inf_b g(x,b)$, we have that  $u \in C(\Omega)$ and
		\begin{equation}
		\begin{array}{ccc}
		D_x^+ g(x,b) & \subset & D^+ u (x) \\
		D^- u(x) & \subset & D_x^- g(x,b)
		\end{array}
		\end{equation}
		for all $b \in M(x)$, with $M$ given as in \eqref{eqn:M-def}. Here $D^\pm_x g(x,b)$ denote the super- and subdifferentials of $g$ in $x$. 
		\label{prop:supsub1}
	\end{proposition}
	
	\begin{proposition}
		Assume that $g$ is as in Proposition \ref{prop:supsub1} and that $g(\cdot,b)$ is differentiable at $x$ uniformly in $b$, that is
		\begin{equation}
		\label{eq:gunifdiff}
		|g(x+h,b)-g(x,b)-D_x g(x,b) \cdot h| \leq |h| \omega_1(|h|)
		\end{equation}
		for $b \in B$, $h$ small and some modulus of continuity $\omega_1$.
		
		Additionally, suppose that
		\begin{equation}
		\label{gbcont}
		\begin{array}{ll}
		b \rightarrow D_x g(x,b) & \mbox{ is continuous } \\
		b \rightarrow g(x,b) & \mbox{ is lower semicontinuous.}
		\end{array}
		\end{equation}
		
		Let us define
		\begin{equation}
		\begin{array}{ccccc}
		Y(x) & = & \{D_x g(x,b) \, | \, b \in M(x) \} . & &
		\end{array}
		\end{equation} 
		
		Then we have that $Y(x) \neq  \emptyset$ and we can characterize the super- and subdifferentials of $u$ in the following way
		\begin{equation}
		\label{eq:semidiffcharac}
		\begin{array}{ccl}
		D^+u(x) & = &  \overline{\mbox{co}} \, Y(x) \\
		D^-u(x) & = & \left\{  \begin{array}{ccl}
		\{y\} & \mbox{if} & Y(x)=\{y\} \\
		\emptyset & \mbox{if} & Y(x) \mbox{ is not a singleton} .
		\end{array}\right.
		\end{array}
		\end{equation}
		\label{prop:supsub2}
	\end{proposition}

	We now apply the previous propositions to the Tukey depth in order to characterize its super- and subdifferentials.
	
	\begin{theorem} \label{thm:semi-diff-equalities}
	  Suppose that $\mu$ has a bounded and continuous density $\rho$ with an open and bounded support $S \subset \mathbb{R}^n$ and let us consider the Tukey depth defined by $T(x) = \inf_{\nu \in S^{n-1}} Z(x,\nu)$, as in (\ref{eqn:Tukey-def-min}).
		Then we have that
		\begin{equation}
		\begin{array}{ccl}
		D^+T(x) & = & \overline{\mbox{co}} \left\{\displaystyle - \nu_i \int_{(\xi-x) \cdot \nu_i =0} \rho(\xi) \,d\mathcal{H}^{n-1}(\xi) \left| \nu_i \in \displaystyle \underset{\nu \in \mathbb{S}^{n-1}}{\argmin}\, Z(x,\nu) \right. \right\} \\ \\
		D^-T(x) & = & \left\{ \begin{array}{ccl}
		\displaystyle - \nu \int_{(\xi-x) \cdot \nu =0} \rho(\xi) \,d\mathcal{H}^{n-1}(\xi) & \mbox{ if } & \underset{ \nu \in \mathbb{S}^{n-1}}{\argmin} \, Z(x,\nu)= \{\nu\} \\ \\
		\emptyset & \mbox{ if }  &  \underset{ \nu \in \mathbb{S}^{n-1}}{\argmin} \, Z(x,\nu) \mbox{ is not a singleton}
		\end{array} \right.
		\end{array}
		\end{equation}
	\end{theorem}
	
	\begin{proof}
		The uniform continuity assumption (\ref{eq:gunifcont}) on $Z(x,\nu)$ is satisfied because $\rho$ is bounded and has bounded support
		\begin{equation}
		\begin{array}{ccl}
		| Z(x,\nu)-Z(y,\nu)  | & = & \displaystyle \left| \int_{(\xi-x) \cdot \nu \geq 0} \rho (\xi) d\xi - \int_{(\xi-y) \cdot \nu \geq 0} \rho (\xi) d \xi \right| = \left| \int_{y \cdot \nu \leq \xi \cdot \nu \leq x \cdot \nu } \rho (\xi) d \xi\right| \\ \\
		& \leq & \displaystyle  M \int_{y \cdot \nu \leq \xi \cdot \nu \leq x \cdot \nu } \mathbf{1}_{S} \, d \xi \lesssim M (\mbox{diam}(S))^{n-1} |(x-y) \cdot \nu|
		\end{array}
		\label{eqn:ZunifContX}
		\end{equation}
		which is uniform in $\nu$, because $\nu \in \mathbb{S}^{n-1}$. Here $M$ is an upper bound for $\rho$ and the symbol $\lesssim $ indicates that the inequality holds up to a constant.

		The uniform differentiability assumption on $Z$ \eqref{eq:gunifdiff} is also satisfied as $\rho$ is compactly supported and bounded.	As derived in \eqref{eqn:DirDerZ}, the derivative of $Z(x,\nu)$ in the direction of a unit vector $h$ is given by

		\begin{equation}
		D_h Z (x,\nu) = - \displaystyle h \cdot \nu \int_{(\xi-x) \cdot \nu =0} \rho(\xi) \,d\mathcal{H}^{n-1}(\xi),
		\end{equation}
		from which we have that 
		\begin{equation}
		D_x Z(x,\nu)= \textcolor{red}{-} \displaystyle \nu \int_{(y-x) \cdot \nu =0} \rho(y) \,d\mathcal{H}^{n-1}(y) \mbox{.}
		\end{equation}	
		We can now show that (\ref{eq:gunifdiff}) holds uniformly in $b$. Note that by the mean value theorem we have that
		\begin{equation}
		Z(x+h,b)= Z(x,b) + D_x Z(\eta,b) \cdot h
		\end{equation}
		for some $\eta= x +c h $ in the segment joining $x$ and $x+h$. Then we have that
		\begin{equation}
		\begin{array}{ccl}
		|Z(x+h,b)-Z(x,b)-D_x Z(x,b) \cdot h| & = & |D_x Z(\eta,b) \cdot h - D_x Z(x,b) \cdot h| \\
		&  \leq & |h| |D_x Z(\eta,b)-D_x Z(x,b)| .
		\end{array}
		\end{equation}
For the right-hand side we can write
		\begin{equation}
		\begin{array}{ccl}
		|D_x Z(\eta,b)-D_x Z(x,b)|  & = & \displaystyle  \left| b \int_{(\xi - \eta) \cdot b = 0} \rho (\xi) \,d\mathcal{H}^{n-1} (\xi) -  b \int_{(\xi - x) \cdot b = 0} \rho (\xi) \,d\mathcal{H}^{n-1} (\xi)\right| \\
		& \leq & |b| \displaystyle \left|\int_{(\zeta-x) \cdot b = 0} \rho (\zeta - x + \eta) \,d\mathcal{H}^{n-1}(\zeta) - \int_{(\xi - x) \cdot b = 0} \rho (\xi) \,d\mathcal{H}^{n-1}(\xi)  \right| \\
		& = & |b| \displaystyle \left|\int_{(\xi-x) \cdot b = 0} \rho (\xi - x + \eta) \,d\mathcal{H}^{n-1}(\xi) - \int_{(\xi - x) \cdot b = 0} \rho (\xi) \,d\mathcal{H}^{n-1}(\xi)  \right| \\
		& = & |b| \displaystyle \left| \int_{(\xi-x) \cdot b = 0} [\rho (\xi - x + \eta) - \rho (\xi)] \,d\mathcal{H}^{n-1} (\xi) \right| ,
		\end{array}
		\end{equation}
		where we have used the change of variables $\zeta = \xi + x - \eta$. Now, if we denote the modulus of continuity of $\rho$ by $\omega_1$ (i.e. $|\rho (y)-\rho(x)| \leq \omega_1 (|y-x|)$) then we have that
		\begin{equation}
		|D_x Z(\eta,b)-D_x Z(x,b)| \leq |b| |h| \omega_1 (ch) (\mbox{diam} (S))^{n-1} \mbox{,}
		\end{equation}
		which is uniform in $b$ as $|b|=1$. This shows that $Z$ is uniformly diferentiable in $b$.
		
		In order to use Theorem \ref{thm:semi-diff-equalities} we only need to show that $Z(x,\nu)$ is lower semicontinuous in $\nu$ and that $D_x Z(x,\nu)$ is continuous in $\nu$. The first result comes from the fact that the mass between two hyperplanes whose normal vectors are close is necessarily small if both the density and its support are bounded. This result is analogous to the estimate \eqref{eqn:ZunifContX}, that states that the mass between two hyperplanes with the same normal varies continuously with the separation between them.
		
		To show that $D_x(x,\nu)$ is continuous in $\nu$, let us take two arbitrary directions $\nu,\mu$. We have
		
		\begin{equation}
		\begin{array}{cl}
		& |D_x Z(x,\nu)- D_x Z(x,\mu)| \\ = & \displaystyle \left| \nu \int_{(\xi-x) \cdot \nu =0} \rho(\xi) \,d\mathcal{H}^{n-1}(\xi) - \mu \int_{(\xi-x) \cdot \mu =0} \rho(\xi) \,d\mathcal{H}^{n-1}(\xi) \right| \\
		= & \displaystyle |\nu \left( \int_{(\xi-x) \cdot \nu =0} \rho(\xi) \,d\mathcal{H}^{n-1}(\xi) - \int_{(\xi-x) \cdot \mu =0} \rho(\xi) \,d\mathcal{H}^{n-1}(\xi) \right) +(\nu-\mu) \int_{(\xi-x) \cdot \mu =0} \rho(\xi) \,d\mathcal{H}^{n-1}(\xi)| \\
		\lesssim & \displaystyle \left| \int_{(\xi-x) \cdot \nu =0} \rho(\xi) \,d\mathcal{H}^{n-1}(\xi) - \int_{(\xi-x) \cdot \mu =0} \rho(\xi) \,d\mathcal{H}^{n-1}(\xi) \right| + |\nu-\mu| M (\mbox{diam}(S))^{n-1} \\
		= & \displaystyle \left| \int_{(\xi-x) \cdot \mu =0} \rho(R_{\mu \nu} (\xi-x)+x) \,d\mathcal{H}^{n-1}(\xi)  - \int_{(\xi-x) \cdot \mu =0} \rho(\xi) \,d\mathcal{H}^{n-1}(\xi) \right| + |\nu-\mu| M (\mbox{diam}(S))^{n-1} \\
		= & \displaystyle \left| \int_{(\xi-x) \cdot \mu =0} (\rho(R_{\mu \nu} (\xi-x)+x)  - \rho(\xi)) \,d\mathcal{H}^{n-1}(\xi) \right| + |\nu-\mu| M (\mbox{diam}(S))^{n-1} 
		\end{array}
		\end{equation}
		where $R_{\mu \nu}$ is a rotation matrix that takes $\mu$ to $\nu$ and $M$ is an upper bound for $\rho$. By the continuity of the mapping $\mu \mapsto R_{\mu \nu} (\xi-x) + x$ and the continuity of $x \mapsto \rho(x)$ we can choose $\mu$ close enough to  $\nu$ so that for a given $\varepsilon > 0$, $|(\rho(R_{\mu \nu} (\xi-x)+x)  - \rho(\xi))| < \varepsilon$. We then have
		\begin{equation}
		|D_x Z(x,\nu)- D_x Z(x,\mu)| \leq \epsilon (\mbox{supp}(\rho))^{n-1}  + |\nu-\mu| M (\mbox{diam}(S))^{n-1},
		\end{equation}
		which means that the mapping $\nu \rightarrow Z(x,\nu)$ is continuous.
		We then have, according to Proposition \ref{prop:supsub2},
		\begin{equation}
		\begin{array}{ccl}
		D^+T(x) & = & \overline{\mbox{co}} \left\{\displaystyle - \nu_i \int_{(\xi-x) \cdot \nu_i =0} \rho(\xi) \,d\mathcal{H}^{n-1}(\xi) \left| \nu_i \in \underset{\nu \in \mathbb{S}^{n-1}}{\argmin} \, Z(x,\nu) \right. \right\} \\ \\
		D^-T(x) & = & \left\{ \begin{array}{ccl}
		- \displaystyle \nu \int_{(\xi-x) \cdot \nu =0} \rho(\xi) \,d\mathcal{H}^{n-1}(\xi) & \mbox{ if } & \underset{ \nu \in \mathbb{S}^{n-1}}{\argmin} \, Z(x,\nu)= \{\nu\} \\ \\
		\emptyset & \mbox{ if }  &  \underset{ \nu \in \mathbb{S}^{n-1}}{\argmin} \, Z(x,\nu) \mbox{ is not a singleton}
		\end{array} \right.
		\end{array}
		\end{equation}
		\label{eqn:TexplicitSemidiffs}
	\end{proof}
	We notice that the formula for $D^-T$ is analogous to the observation in one dimension that the Tukey depth has no ``downward pointing corners''.  It follows from the definitions of super- and subdifferential \eqref{eqn:semidiffs} that if $D^+T$ and $D^-T$ are nonempty, then T is differentiable (Lemma 1.8 in \cite{bardi2008optimal}). The explicit formula for $D^-T$ \eqref{eqn:TexplicitSemidiffs} then implies that $T$ is a supersolution of \eqref{eqn:HequationTukey}: If $D^-T$ is nonempty at some point $x$, then $T$ is differentiable and therefore the condition for $T$ to be a supersolution \eqref{def:viscositySol} is satisfied with an equality. On the other hand, if $D^-T$ is empty at some point $x$ then the condition for $T$ to be a supersolution \eqref{eqn:semidiffs} is trivially satisfied at $x$. Thus, we have the following corollary of Theorem \eqref{thm:semi-diff-equalities}. 
	\begin{corollary}
		Under the assumptions of Theorem \ref{th:exuniq}, the Tukey depth $T(x)$ is a supersolution of Equation \eqref{eqn:HequationTukey}.
	\label{cor:TukeySuper}
	\end{corollary}
	
	The previous corollary reveals an important property: that the viscosity solution of \eqref{eqn:TukeyEqn} will necessarily give a lower bound for the Tukey depth. This holds without any structural assumptions on the distribution (besides weak smoothness assumptions). In certain situations there are high fidelity algorithms for solving differential equations of the form \eqref{eqn:TukeyEqn}; we provide some illustrations in the last section of the paper. Although there are many computational considerations in obtaining a solution to \eqref{eqn:TukeyEqn}, especially in the context of data clouds, Corollary \ref{cor:TukeySuper} suggests a promising new avenue for constructing lower bounds on Tukey depths.

\subsection{Tukey depths on convex domains are subsolutions}

Up to this point, we have demonstrated that the Tukey depth is a supersolution to the Hamilton-Jacobi equation \eqref{eqn:TukeyEqn}, and that there exists a unique viscosity solution to the equation. However, it is not yet clear under what conditions the Tukey depth is equal to the viscosity solution. The following two examples, one positive and one negative, are illustrative.

\begin{example}
  Consider a uniform distribution on $[0,1]\times [0,1]$. The Tukey depth is explicitly given by \cite{nagy2019halfspace}
  \[
    T(x,y) = \min 2\left( xy,(1-x)y,x(1-y),(1-x)(1-y) \right).
  \]
  We may directly compute that $T$ is differentiable on the set $x\neq 1/2,y\neq 1/2$, and has gradient given by
  \[
    \nabla T = \begin{cases} 2(y,x) &\text{ for } x<1/2,y<1/2 \\
      2(1-y,-x) &\text{ for } x < 1/2,y>1/2 \\
      2(-y,1-x) & \text{ for } x >1/2, y<1/2 \\
    2(y-1,x-1) &\text{ for } x > 1/2, y> 1/2. \end{cases}
  \]
  The superdifferential on the lines $x = 1/2$ and $y=1/2$ is given by the convex hull of the gradients approaching those points (see Theorem \ref{thm:semi-diff-equalities}).

  In order to prove that $T$ is a subsolution of the HJB equation, it suffices to show, at points where either $x = 1/2$ or $y=1/2$, that for any $v \in D^+ T$ we have that
  \[
    |v| - \int_{v \cdot x = v\cdot y} \rho(y) \,\,d\mathcal{H}^{n-1}(y) \leq 0.
  \]
  By symmetry, it suffices to consider the case where $x=1/2$ and $0 < y < 1/2$, and the case $x = 1/2, y=1/2$. In the first case, we have that $D^+ T = co( 2(y,1/2),2(-y,1/2))$. Thus to verify the subsolution inequality at the base point $(1/2,y)$, we need only consider $v = (2t,1)$, for $t \in (-y,y)$. We then observe that the integral term in the subsolution inequality, with a particular choice of $v = (t,1/2)$, will be equal to the length of the intersection of the square with the line through $(1/2,y)$ with direction given by $(1/2,-t)$. This length will be given by the distance between $(0,y-t)$ to $(1,y+t)$ (i.e. the boundary points of that line), and that distance is hence given by $\sqrt{1+(2t)^2}$. But this is exactly the length of $v$, which implies that the subsolution inequality is actually an equality for all permissible $v$'s. 

    For the point $(1/2,1/2)$, we have that $D^+ T = co((\pm 1, \pm 1) )$. Using exactly the same computation as above, we have that the sub-solution inequality is satisfied (with equality) for any point in $\partial D^+ T$. For $v$ on the interior of $D^+ T$ the inequality also holds, without equality, as the norm term will have decreased (as compared to a point in $\partial D^+ T$), while the negative integral term will have remained unchanged.
\end{example}

The previous example suggests that, in at least some cases, the Tukey depth does coincide with the viscosity solution of the HJB equation. We notice that the example does have non-trivial singularities along lines, and hence it seems unlikely to be a coincidence. However, the following example demonstrates that the Tukey depth is not always equal to the viscosity solution.

\begin{example}\label{ex:not-visc-sol}
	\begin{figure}[h]
		\centering
		\includegraphics[width=0.8\textwidth]{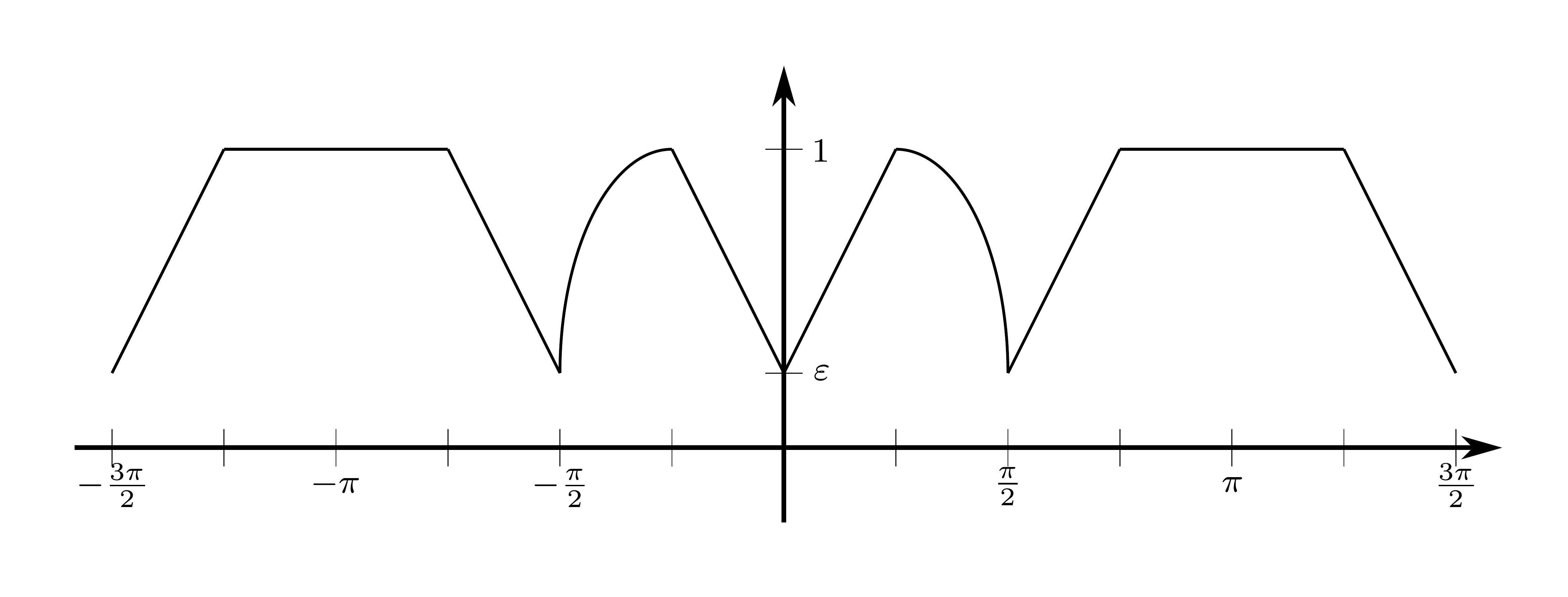}
		\caption{Definition of the function $\psi$ in Example \ref{ex:not-visc-sol}.}
		\label{fig:counterexample}
	\end{figure}
	
		In $\mathbb{R}^2$, define a function $\rho$ on $B(0,1)$ in polar coordinates as
		\begin{equation}
		\rho(r,\theta)=C \, r \, \psi(\theta) ,
		\end{equation} 
		where $\psi$ is the $2 \pi$-periodic function described in figure \ref{fig:counterexample} for some small $\varepsilon >0$ and $C$ is the normalization constant that makes $\rho$ a probability density on $B(0,1)$. We shall see that the Tukey depth $T$ is not a subsolution of the HJB equation \eqref{eqn:HequationTukey}.
		
		We can identify a vector $\nu \in \mathbb{S}^2$ by the angle $\phi$ it forms with the $x$-axis. For the mass on the halfspace defined by $\nu$ we have the expression
		\begin{equation}
		Z(0,\nu)= \int_{\phi -\frac{\pi}{2}}^{\phi +\frac{\pi}{2}} \int_{0}^{1} C \,r \rho(r,\theta) \, dr \, d \theta = \frac{C}{3} \int_{\phi -\frac{\pi}{2}}^{\phi +\frac{\pi}{2}} \psi(\theta) \, d \theta
		\end{equation}	
		where $\approx$ denotes equality up to a constant. By the continuity of $\psi$, $Z$ is differentiable with respect to $\phi$ and critical points of $Z$ come from finding the zeros of
		\begin{equation}
		\frac{d Z}{d \phi} = \frac{C}{3} \left[ \psi \left(\phi + \frac{\pi}{2} \right) - \psi \left(\phi - \frac{\pi}{2} \right) \right].
		\end{equation}
		Equivalently, we want to find pairs of points in the domain of $\psi$ that are $\pi$ units apart and at which $\psi$ attains the same value. Checking the sign of $\frac{d Z}{d \phi}$ with the help of the graph of $\psi$ we can see that, up to multiples of $2 \pi$, there are three distinct points corresponding to local minima: $-\frac{\pi}{4}$, $\frac{\pi}{4}$, and $\pi$. By comparing areas under the graph of  $\psi$, we find that $Z(-\frac{\pi}{4})=Z(\frac{\pi}{4}) < Z(\pi)$. According to theorem \ref{thm:semi-diff-equalities}, both $\nu_1= C \, (\cos (-\frac{\pi}{4}), \sin (-\frac{\pi}{4}))$ and $\nu_2= C\,(\cos(\frac{\pi}{4}),\sin(\frac{\pi}{4}))$ are in $\partial^+T(0)$ and therefore $v=(\frac{C}{\sqrt{2}},0)$, which is in the convex hull of $\nu_1$ and $\nu_2$, is also in $\partial^+T(0)$. However, we have that
		\begin{equation}
		H(0,v)= |v|- \int_{y \cdot v=0} \rho(y) \,d\mathcal{H}^{n-1} (y) =\frac{C}{\sqrt{2}}- \varepsilon \, C .
		\end{equation}
		Because $\varepsilon$ can be made as small as desired in the definition of $\psi$, $T$ cannot be a subsolution of \eqref{eqn:HequationTukey}.
\end{example}

In light of the previous examples, it would be desirable to characterize sufficient conditions that guarantee that the Tukey depth and the viscosity solution of \eqref{eqn:HequationTukey} match. Finding such a sufficient condition is not simple, and is a topic of our current work. However, we notice that in the previous example, the failure of the Tukey depth to be a viscosity solution is linked with the fact that the density has non-convexity: namely it has a ``valley'' between two regions of higher density. This contrasts with the fact that the uniform density on the square does not have such valleys. One then may guess that convexity may be connected to a sufficient condition that Tukey depths and viscosity solutions match. The following theorem provides a first step in this direction.

\begin{theorem}
  Suppose that $\rho$ is uniformly distributed on a convex domain in $\R^2$. Then the Tukey depth and the viscosity solutions are equal.
  \label{thm:convex-sub-sol}
\end{theorem}

The proof that we present centers on a direct geometric argument. We now state several lemmas which will be needed in proving the previous theorem.

\begin{lemma}
  Let $\rho$ be continuous. The viscosity solution to Equation \eqref{eqn:TukeyEqn} is invariant under invertible affine transformations.
  \label{lem:visc-affine-invariant}
\end{lemma}

\begin{proof}
  If $u$ is a viscosity solution corresponding to a density $\rho$, and $L$ is an invertible affine transformation, we let $\tilde \rho(\tilde{x}) = \rho(L(\tilde{x})) det(D L)$, and define $\tilde u(\tilde{x}) = u(L(\tilde{x}))$. As in the rest of the paper, we are treating all vectors as column vectors. Using the definition of the upper and lower semi-differentials (see, e.g., Chapter II in \cite{bardi2008optimal}), we then have that $D^\pm \tilde u = (DL)^T \, D^\pm u$, meaning that if $p \in D^+ u$ then $(DL)^T p \in D^+ \tilde u$. Hence for a given $\tilde p \in D^+ \tilde u$ we have
  \[
    |\tilde p| = |(DL)^T p|
  \]
  for a $p \in D^+u$ at some point $x$.
  
  On the other hand, using a change of variables $\tilde \xi = L^{-1}(\xi)$, and the continuity of $\rho$, we may compute
  \begin{align*}
    \int_{\tilde p\cdot(\tilde \xi-L^{-1}(x))=0} \tilde \rho(\tilde \xi)\,d\mathcal{H}^{n-1}(\tilde\xi) &= \lim_{h \to 0} \frac{|\tilde p|}{2h} \int_{|\tilde p\cdot(\tilde \xi-L^{-1}(x))|\leq h} \tilde \rho(\tilde \xi)\,d \mathcal{L}^n(\tilde\xi) \\
    &  = \lim_{h \to 0} \frac{|\tilde p|}{2h} \int_{|\tilde p\cdot(\tilde \xi-L^{-1}(x))|\leq h} \rho(L(\tilde \xi)) \det (DL)\,d \mathcal{L}^n(\tilde\xi) \\
    &= \lim_{h \to 0} \frac{|\tilde p|}{|p|}\frac{|p|}{2h} \int_{| ((DL)^T p) \cdot( L^{-1}(\xi)-L^{-1}(x))|\leq h} \rho( \xi)\,d \mathcal{L}^n(\xi) \\
    &= \lim_{h \to 0} \frac{|(DL)^T p|}{|p|}\frac{|p|}{2h} \int_{| p \cdot((DL) (DL)^{-1}(\xi-x))|\leq h} \rho( \xi)\,d \mathcal{L}^n(\xi) \\
    &= \lim_{h \to 0} \frac{|(DL)^T p|}{|p|}\frac{|p|}{2h} \int_{| p \cdot(\xi-x)|\leq h} \rho( \xi)\,d \mathcal{L}^n(\xi) \\
     &= \frac{|(DL)^T p|}{|p|} \int_{p \cdot (\xi-x)=0} \rho(\xi)  \,d\mathcal{H}^{n-1}(\xi).
  \end{align*}
  As $u$ is a subsolution, this implies that
  \[
    |\tilde p| - \int_{\tilde p\cdot(\tilde \xi-L(x))=0} \tilde \rho(\tilde \xi)\,d\tilde\xi = \frac{|(DL)^T p|}{|p|} \left(|p| - \int_{p \cdot (\xi-x)=0} \rho(\xi)  \,d\mathcal{H}^{n-1}(\xi)\right) \leq 0
  \]
  An analogous argument applies for supersolutions, which concludes the proof.
\end{proof}

We recall that an extreme point $w$ of a closed, convex set $C$ is a point so that $w$ cannot be represented as a convex combination of two other points in $C$.

\begin{lemma}
  Let $w$ be a non-zero extreme point of $D^+T(x)$. Then we have that
  \[
    |w|-\int_{w\cdot(\xi-x) = 0} \rho(\xi) \,d\mathcal{H}^{n-1}(\xi) = 0.
  \]
  \label{lem:extreme-points}
\end{lemma}

\begin{proof}
  By Lemma 1.8 in \cite{bardi2008optimal}, the set of points where $T$ is differentiable is dense, and at any such point the PDE \eqref{eqn:TukeyEqn} is satisfied by the argument in Section \ref{sec:Derivation}. As $w$ is an extreme point of $D^+T(x)$, by the definition of $D^+T$ there must exist a sequence of points $x_n \to x$ and $DT(x_n) = w_n \to w$. Taking limits and using the continuity of the integral term, which holds as $w$ is non-zero, we then have that the desired equality holds.
\end{proof}

\begin{lemma}
  In $\R^2$, any hyperplane (with normal $\nu$) minimizing the halfspace depth through a point $x$ is \emph{balanced}, in the sense that
  \[
    \int_0^\infty \rho(x +  r\nu^\perp) r \,dr = \int_{-\infty}^0 \rho( x + r\nu^\perp) r \,dr,
  \]
  where $\nu^\perp$ is a normal vector orthogonal to $\nu$.  In particular, for uniform densities in dimension two we have that the probability mass through the line bounding a minimal halfspace is the same on either side of $x$; or in other words the line has the same length on either side of $x$.
  \label{lem:balanced-lines}
\end{lemma}

\begin{proof}
  Given normal vector $\nu$, we may express the probability density assigned to a halfspace in terms of polar coordinates by
  \[
    H(x,\nu) = \int_{(x-y)\cdot \nu \geq 0} \rho(y)\,dy = \int_{\theta-\pi/2}^{\theta + \pi/2} \int_0^\infty \rho(r,\theta) r \,dr \,d\theta,
  \]
  where here $\theta$ represents the angle corresponding to $\nu$. In turn, taking a derivative with respect to $\theta$ and setting it to zero gives the desired equality. Applying that equality in two dimensions to uniform densities and integrating immediately gives the second result:
\end{proof}

\begin{lemma}\label{lem:boundary-convex-2D}
  In two dimensions, any point in the boundary of a closed, bounded convex set is either an extreme point or can be represented as a convex combination of two extreme points.
\end{lemma}

\begin{proof}
  In two dimensions, a supporting hyperplane to a convex set is just a line, and so any point in the boundary of that convex set could only possibly be a linear combination of two points along that line. The intersection of that line with the convex set will consitute a closed line segment, whose endpoints will be extreme points, which establishes the lemma.
\end{proof}

We can now prove the theorem:

\begin{proof}[Proof of Theorem \ref{thm:convex-sub-sol}]

  To prove the theorem, we simply need to show that any $p \in D^+T(x)$ satisfies
  \begin{equation}\label{eqn:sub-sol-proof}
    |p| - \int_{p \cdot (\xi-x)} \rho(\xi) \,\,d\mathcal{H}^{n-1}(\xi) \leq 0.
  \end{equation}
  By Theorem \ref{thm:semi-diff-equalities} we have that this inequality (actually with equality) is satisfied at any extreme point of $D^+T$, and that those extreme points correspond to the directions minimizing the halfspace measure. Thus the only case we need to verify is the case where $D^+T$ contains multiple points, and we only need to verify it for points that are not extreme points.

  We claim that for any two distinct extreme points $p_1,p_2$ of $D^+T(x)$, we have that $tp_1 + (1-t)p_2$ satisfies \eqref{eqn:sub-sol-proof} at the point $x$, for any $0 \leq t \leq 1$. By taking an appropriate affine transformation, we may assume that $p_1 = 2 e_1, p_2 = 2 e_2$ and that $x = 0$. After applying this transformation by using Lemma \ref{lem:balanced-lines} and the fact that \eqref{eqn:sub-sol-proof} holds with equality at $p_1,p_2$, we know that $\partial \Omega \cap \{x=0\} = \{0\}\times [-1,1]$ and $\partial \Omega \cap \{y=0\} = [-1,1]\times\{0\}$. The desired estimate can be described geometrically as follows: given $0 \leq t \leq 1$, is $|tp_1 + (1-t)p_2|$ smaller than the length of the intersection of the line through the origin in the direction $tp_1 + (1-t)p_2$ with a convex domain containing the lines $\{0\}\times [-1,1]$ and $[-1,1]\times\{0\}$; see Figure \ref{fig:convex-vs-triangle}. This, however, is always true: in particular the smallest such convex domain must contain the triangles $( (0,0), (1,0), (0,1))$ and $( (0,0),(-1,0),(0,-1))$, and the length of the intersection of the line with those triangles is exactly equal to $|tp_1 + (1-t)p_2|$, see Figure \ref{fig:quad-vs-triangle}.

  \begin{figure}[h]
    \centering
\includegraphics[width=.25\columnwidth]{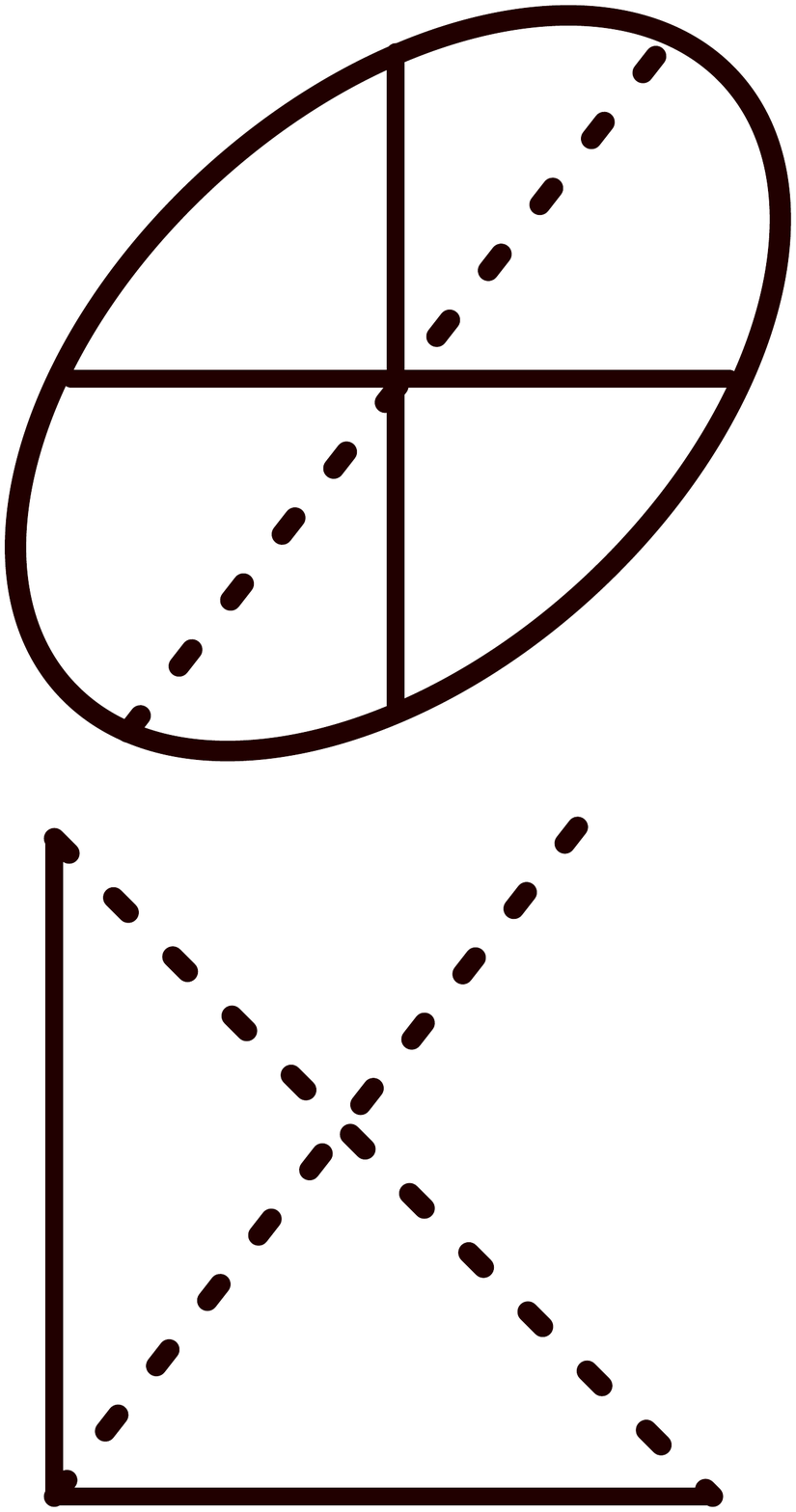}
\caption{Comparing the lengths of the lines inside the convex domain and the triangle, both pointing in the same direction.}
    \label{fig:convex-vs-triangle}
  \end{figure}

  \begin{figure}[h]
    \centering
\includegraphics[width=.25\columnwidth]{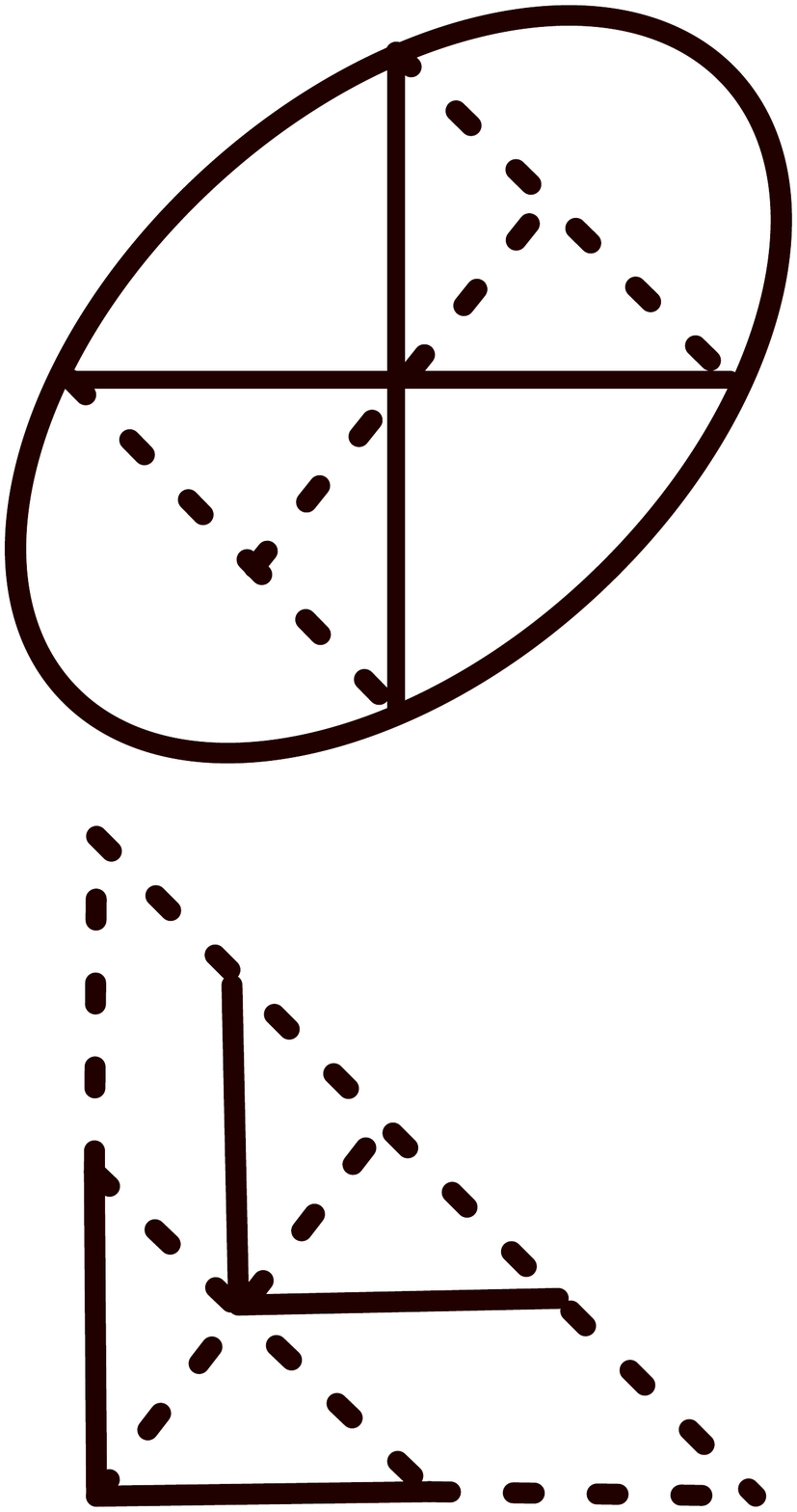}
\caption{Comparing the lengths of the lines inside the triangles inscribed in the convex domain and the triangle, both pointing in the same direction. The second illustration shows the lower triangle reflected in the other direction, and given the equal aspect ratios of the triangles, the equality is immediate.}
    \label{fig:quad-vs-triangle}
  \end{figure}

  Now, we claim that, in two dimensions, the previous claim is sufficient to establish \eqref{eqn:sub-sol-proof} for all $p \in D^+T(x)$. Suppose that $\tilde p \in D^+T(x)$ is of maximal magnitude for a specific direction, meaning $\sup\{s : s \tilde p \in D^+T(x) \} = 1$. Then either $\tilde p$ is an extreme point of $D^+T(x)$ or it can be represented as a convex combination of two other extreme points by Lemma \ref{lem:boundary-convex-2D}. The previous claim then implies that $\tilde p$ must satisfy \eqref{eqn:sub-sol-proof}. We then notice that the set where \eqref{eqn:sub-sol-proof} is satisfied is star-shaped with respect to the origin, meaning that if $\tilde p$ satisfies \eqref{eqn:sub-sol-proof} then $s\tilde p$ for any $s \in [-1,1]$: this is a simple consequence of the fact that $|\cdot|$ is one-homogeneous while the integral term is zero homogeneous. This then implies that \eqref{eqn:sub-sol-proof} holds for all $p \in D^+T(x)$, which concludes the proof.
\end{proof}

\begin{remark}
  A natural question is how to relax the assumptions of the previous theorem. In particular, the restriction to two dimensions and uniform measures seems overly strict. Example \ref{ex:not-visc-sol} suggests that some type of convexity is a necessary assumption: we suspect that the theorem should continue to hold for any log-concave measure in any dimension. However, the transparent geometric proof that we present here is not applicable in higher dimension, and other tools would need to be developed. For brevity and accessibility we restrict our attention to the present setting and leave extension to future work.
\end{remark}

\begin{remark}
  We also notice that Theorem \ref{thm:convex-sub-sol} can be restated in terms of convex floating bodies \cite{nagy2019halfspace}: namely that convex floating bodies in two dimensions are characterized by the viscosity solution to \eqref{eqn:TukeyEqn}. We remark that there is a regularity theory for convex floating bodies (see Proposition 13 in \cite{nagy2019halfspace}): namely if a strictly convex, centrally symmetric body has $C^1$ boundary then the floating body (which we identify as the level sets of the Tukey depth) has $C^2$ boundary. In that case the viscosity theory is unnecessary. However, the theory we give here applies for convex bodies which are not regular, and for which the convex floating body does not have smooth boundary.
\end{remark}

\section{Numerical Illustrations}\label{sec:illustrations}
In this section we demonstrate some numerical solutions to the HJB equations we derived in Section \ref{sec:Derivation}. The purpose of this section is to provide illustration and not to provide a full account or analysis of such numerical methods. Indeed, more detailed analysis and efficient implementation of the numerics for these equations is the topic of current work.

Figures \ref{fig:square} , \ref{fig:triangle}, \ref{fig:gaussian}, and \ref{fig:cauchy} illustrate a numerical approximation of the Tukey depth using the differential equation \eqref{eqn:TukeyEqn}, for four different input distributions. The first two cases deal with uniform distributions on the square and triangle, both of which are covered in Theorem \ref{thm:convex-sub-sol}. The third case, namely a Gaussian in two dimensions, does not fall within that theory, but clearly is in a regime in which the Tukey depth is a viscosity solution, as the solution is smooth except at its maximum. The fourth case, namely a bivariate Cauchy distribution, does not have the same degree of convexity as the other examples do, but the numerical approximations of the viscosity solution still match the Tukey depth. We also highlight the fact that qualitatively, the solutions match the true solutions extremely well, resolving the singularities in matching ways.

Each of these solutions was generated by using a grid-based representation of the PDE solution. The approximate PDE solution was found by using a fast marching method \cite{sethian1999level,sethian1999fast} which incorporates an upwind representation of derivatives that is consistent with viscosity solutions. Details about the grid size, and the method for approximating the integral on the right-hand side of the differential equation, are given in the captions for the plots.

We emphasize that these results are primarily illustrative and preliminary. Extending these numerical methods to large data sets is a challenging task. Detailed analysis of these numerical methods is also non-trivial. However, these illustrations provide evidence that the differential equation approach can be quite effective, and thus warrants further study.

\begin{figure}
\centering
     \begin{subfigure}[b]{0.49\textwidth}
         \centering
         \includegraphics[width=\textwidth]{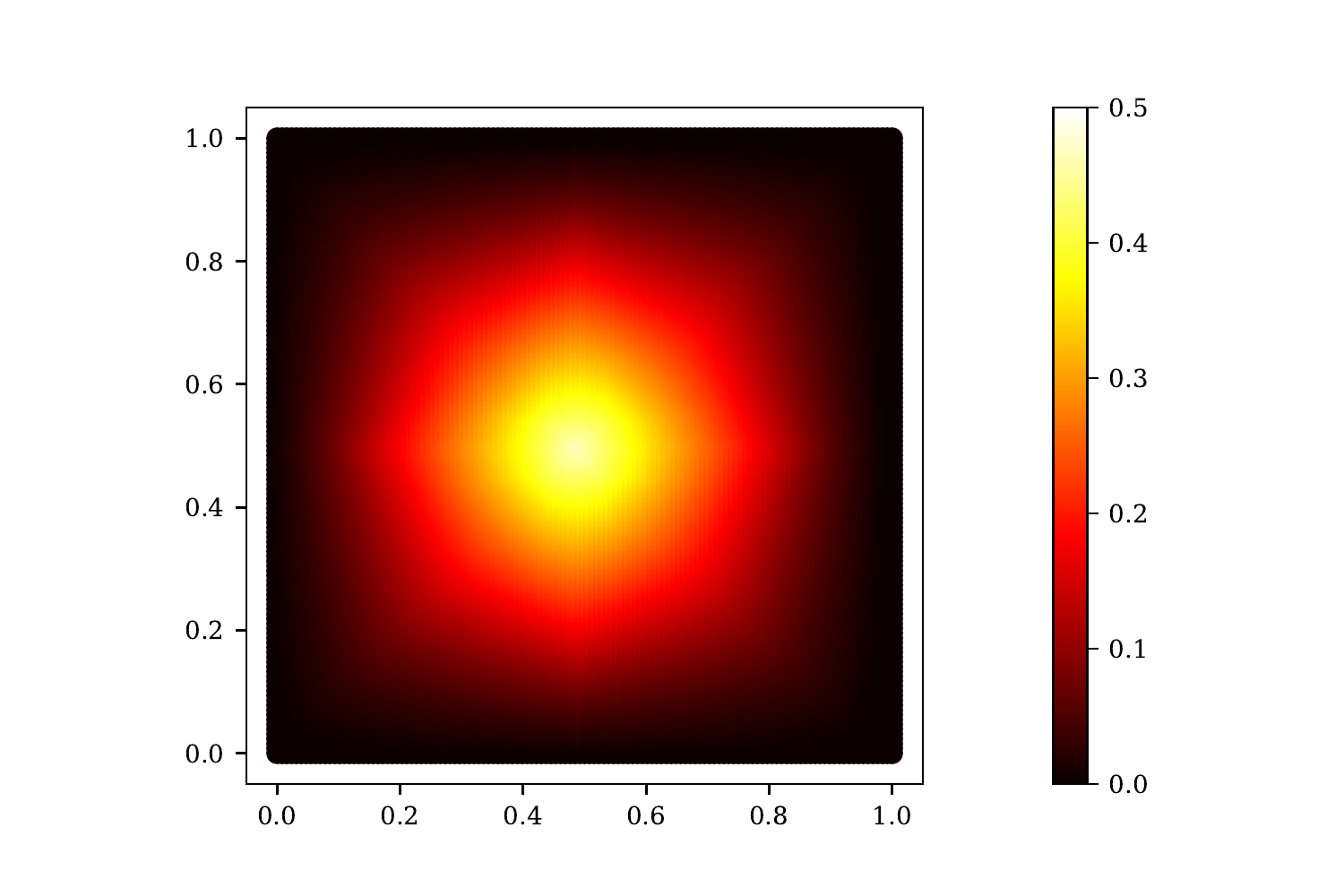}
         \caption{}
         \label{fig:square-vals}
     \end{subfigure}
     \hfill
     \begin{subfigure}[b]{0.49\textwidth}
         \centering
         \includegraphics[width=\textwidth]{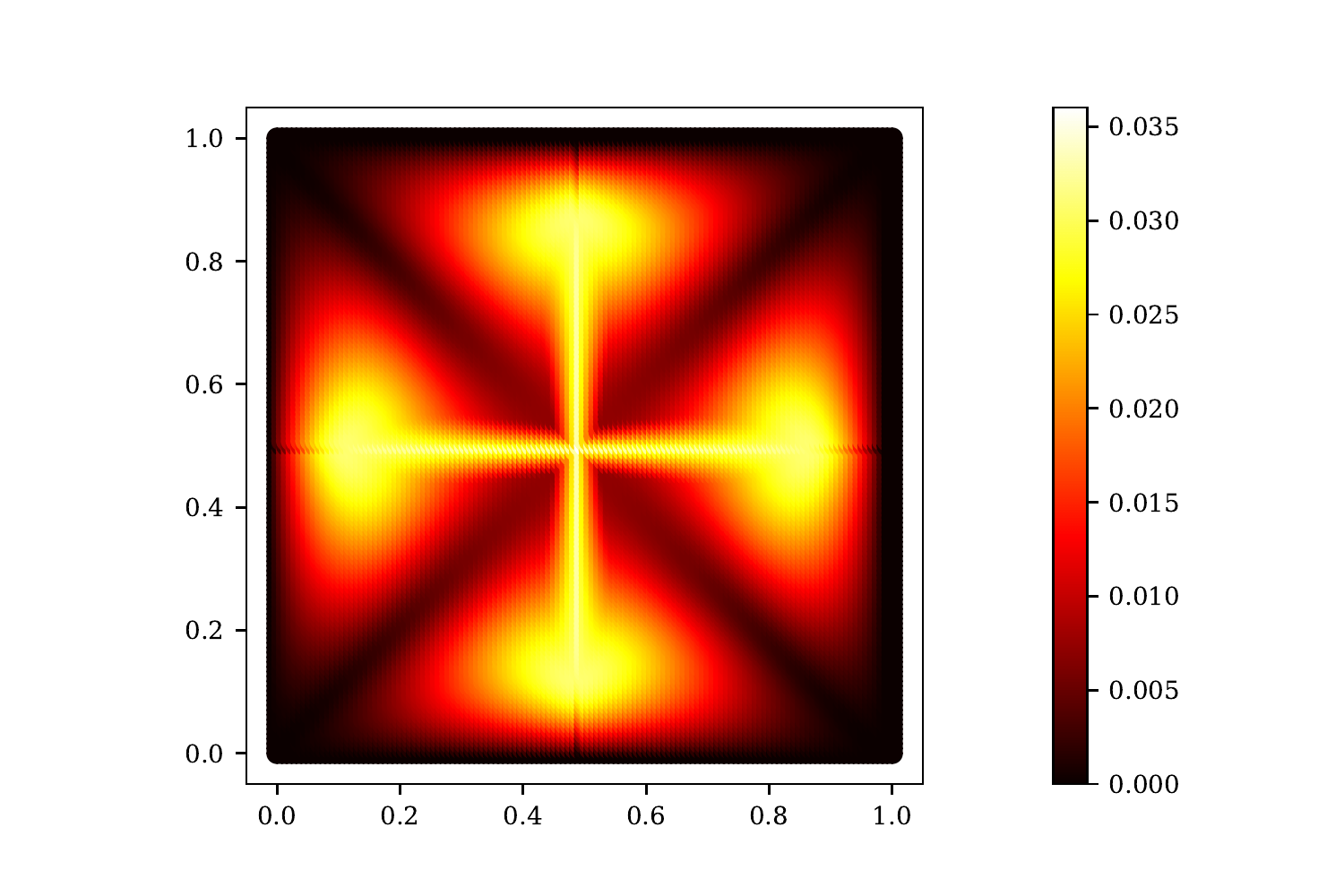}
         \caption{}
         \label{fig:square-errors}
     \end{subfigure}
     \\
     \begin{subfigure}[b]{0.49\textwidth}
         \centering
         \includegraphics[width=\textwidth]{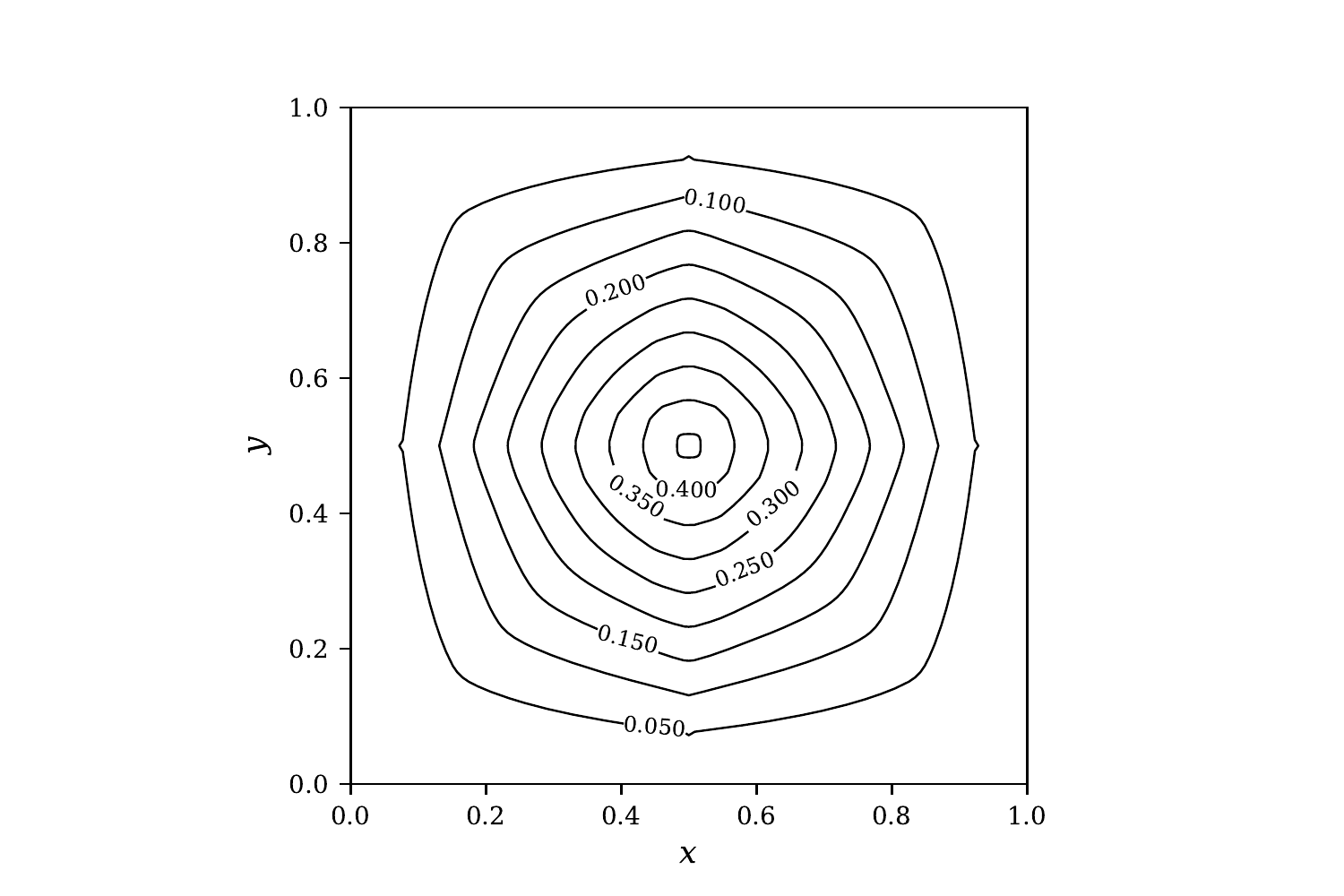}
         \caption{}
         \label{fig:square-calc-contours}
     \end{subfigure}
     \hfill
     \begin{subfigure}[b]{0.49\textwidth}
         \centering
         \includegraphics[width=\textwidth]{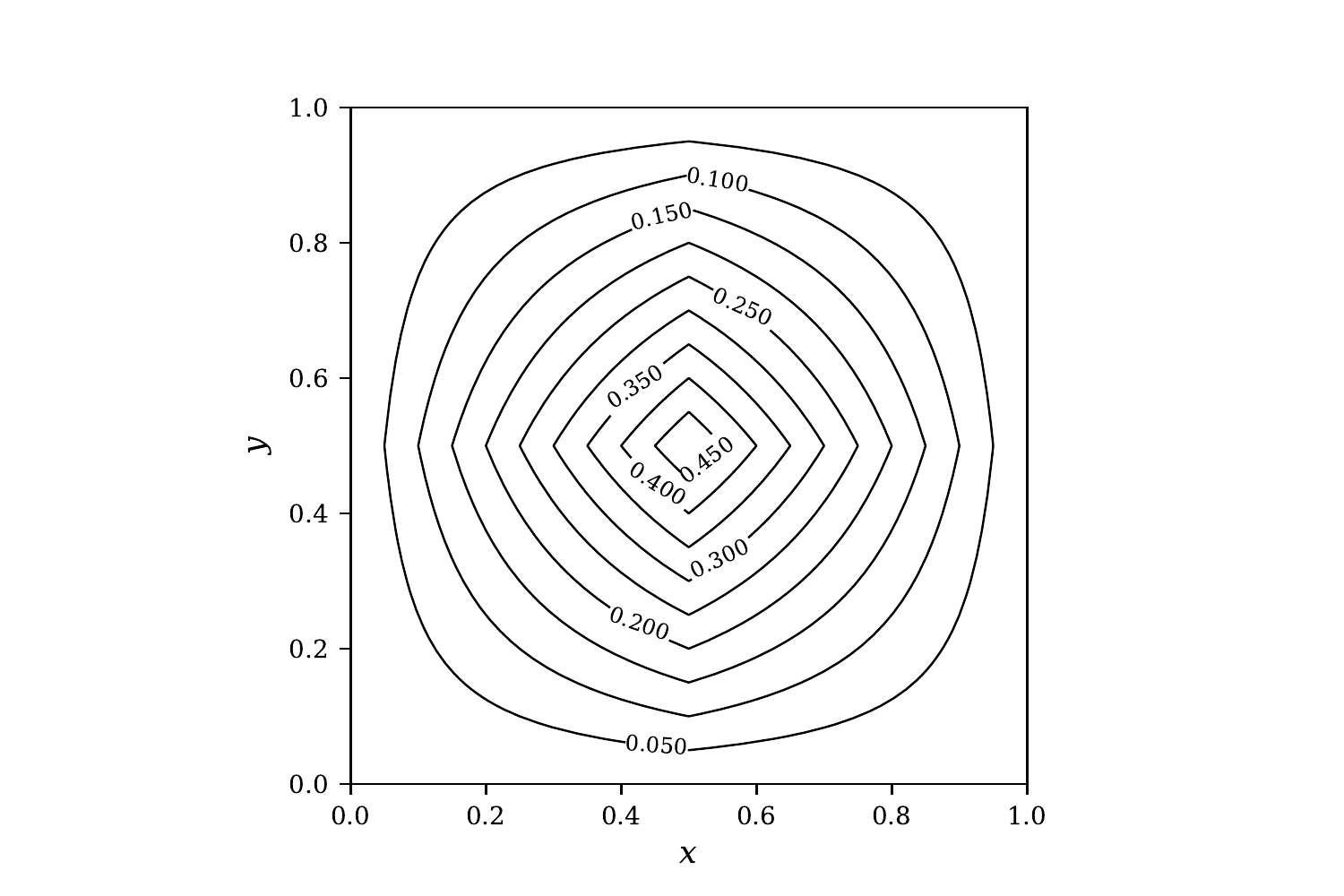}
         \caption{}
         \label{fig:square-actual-contours}
     \end{subfigure}
        \caption{Approximation of the Tukey depth using Equation \eqref{eqn:TukeyEqn}, for a uniform density on a square. Figure \ref{fig:square-vals} shows the computed values, Figure \ref{fig:square-errors} shows the difference with the exact solution, Figure \ref{fig:square-calc-contours} shows the level sets of the computed values and Figure \ref{fig:square-actual-contours} shows the level sets of the true contours. Here the (square) grid spacing is $1/128$, and the integral term in the right-hand side is computed analytically.}
        \label{fig:square}
\end{figure}

\begin{figure}
\centering
     \begin{subfigure}[b]{0.49\textwidth}
         \centering
         \includegraphics[width=\textwidth]{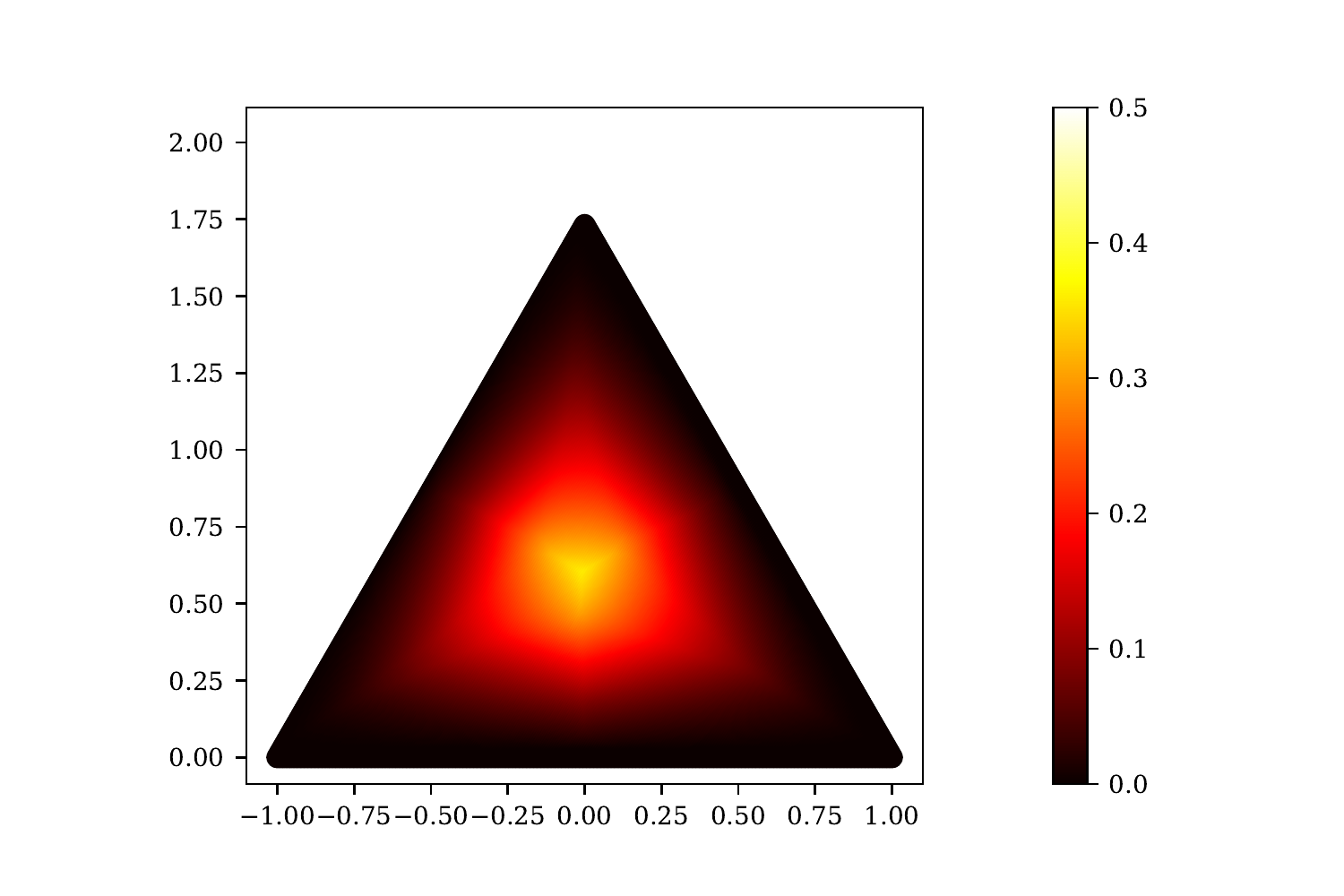}
         \caption{}
         \label{fig:triangle-vals}
     \end{subfigure}
     \hfill
     \begin{subfigure}[b]{0.49\textwidth}
         \centering
         \includegraphics[width=\textwidth]{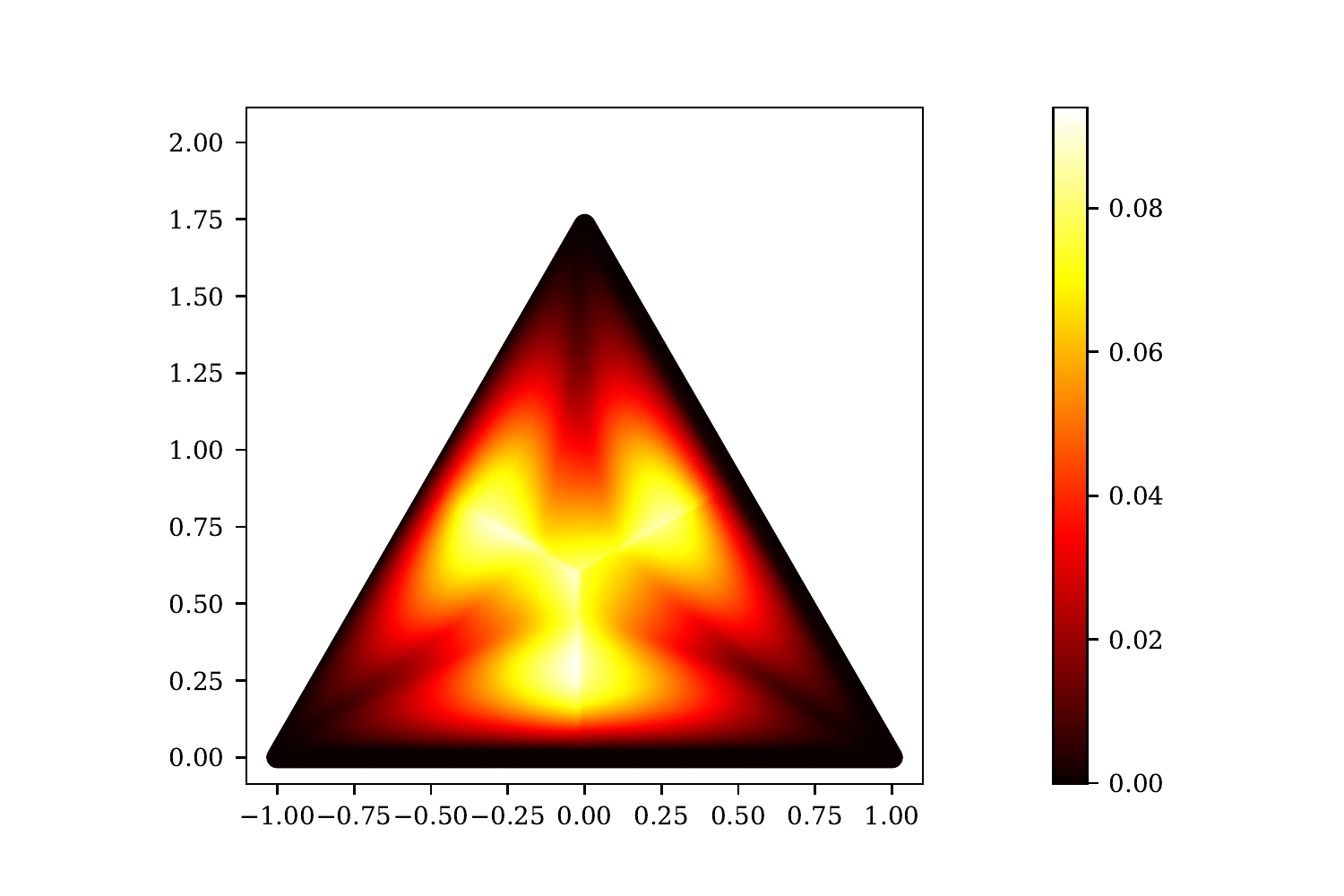}
         \caption{}
         \label{fig:triangle-errors}
     \end{subfigure}
     \\
     \begin{subfigure}[b]{0.49\textwidth}
         \centering
         \includegraphics[width=\textwidth]{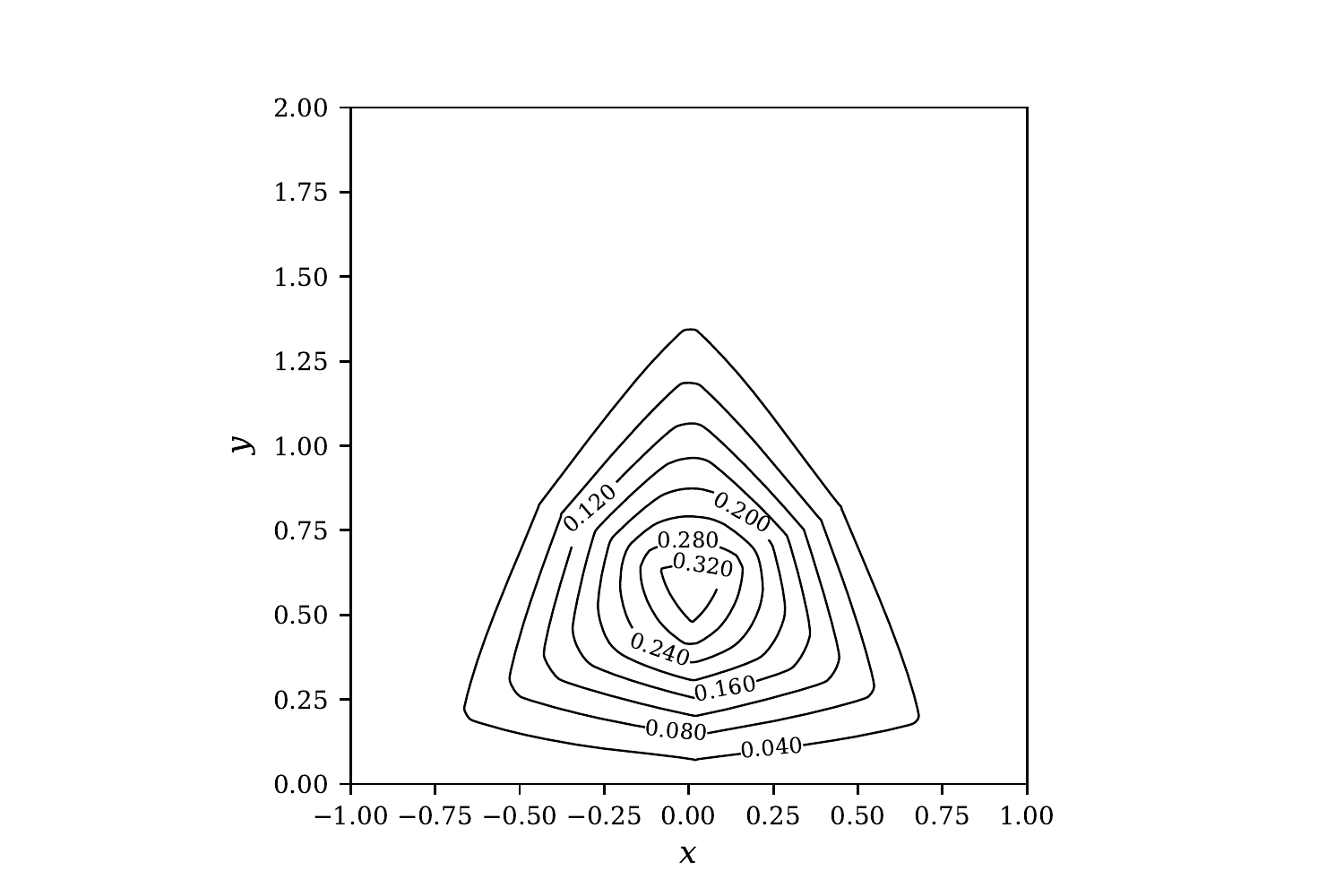}
         \caption{}
         \label{fig:triangle-calc-contours}
     \end{subfigure}
     \hfill
     \begin{subfigure}[b]{0.49\textwidth}
         \centering
         \includegraphics[width=\textwidth]{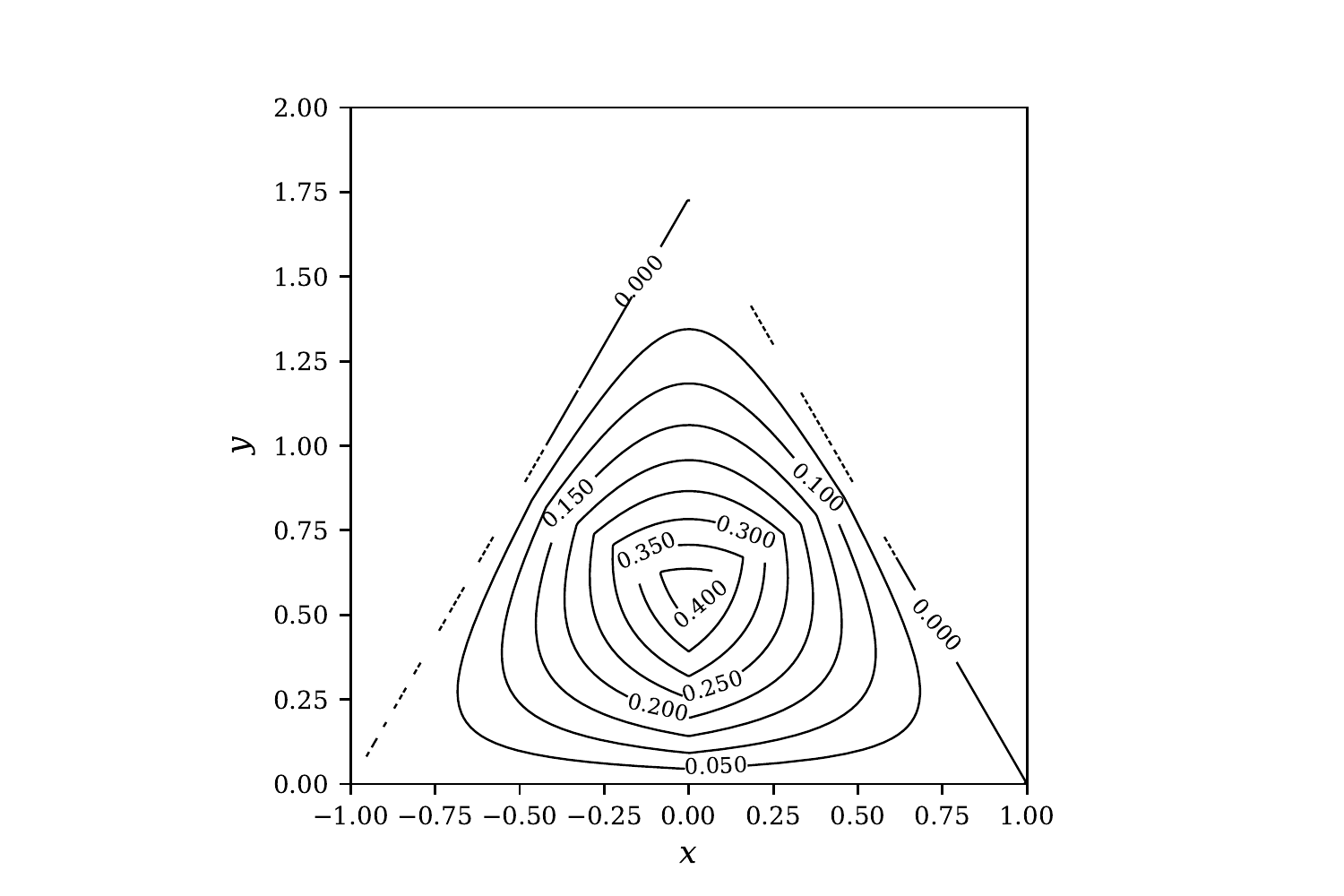}
         \caption{}
         \label{fig:triangle-actual-contours}
     \end{subfigure}
        \caption{Approximation of the Tukey depth using Equation \eqref{eqn:TukeyEqn}, for a uniform density on a triangle. Figure \ref{fig:triangle-vals} shows the computed values, Figure \ref{fig:triangle-errors} shows the difference with the exact solution, Figure \ref{fig:triangle-calc-contours} shows the level sets of the computed values and Figure \ref{fig:triangle-actual-contours} shows the level sets of the true contours. Here the (triangular) grid spacing is $1/128$, and the right-hand side is computed by approximating the integral using a random sample (size 12k) from the domain.}
        \label{fig:triangle}
\end{figure}

\begin{figure}
\centering
     \begin{subfigure}[b]{0.49\textwidth}
         \centering
         \includegraphics[width=\textwidth]{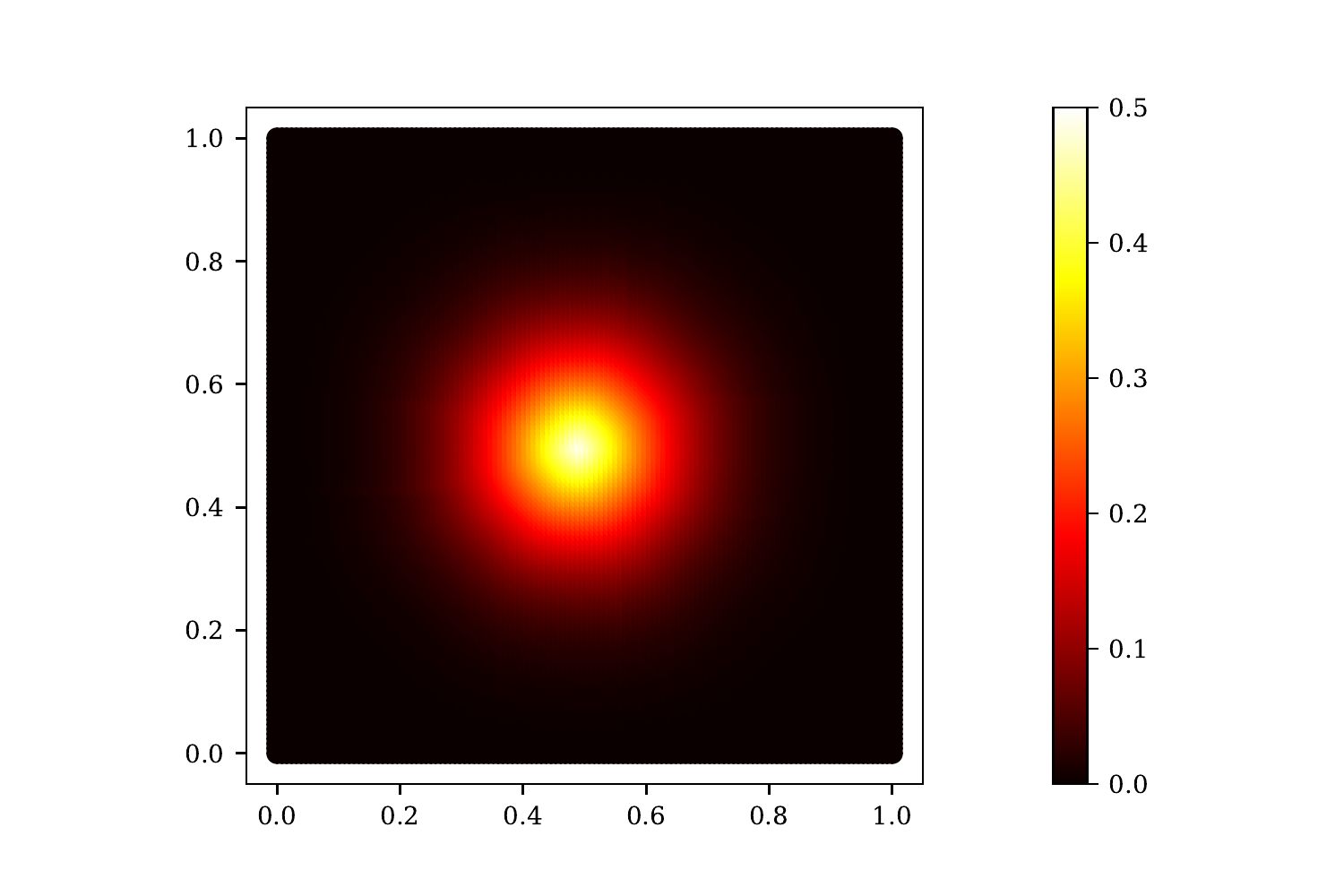}
         \caption{}
         \label{fig:gaussian-vals}
     \end{subfigure}
     \hfill
     \begin{subfigure}[b]{0.49\textwidth}
         \centering
         \includegraphics[width=\textwidth]{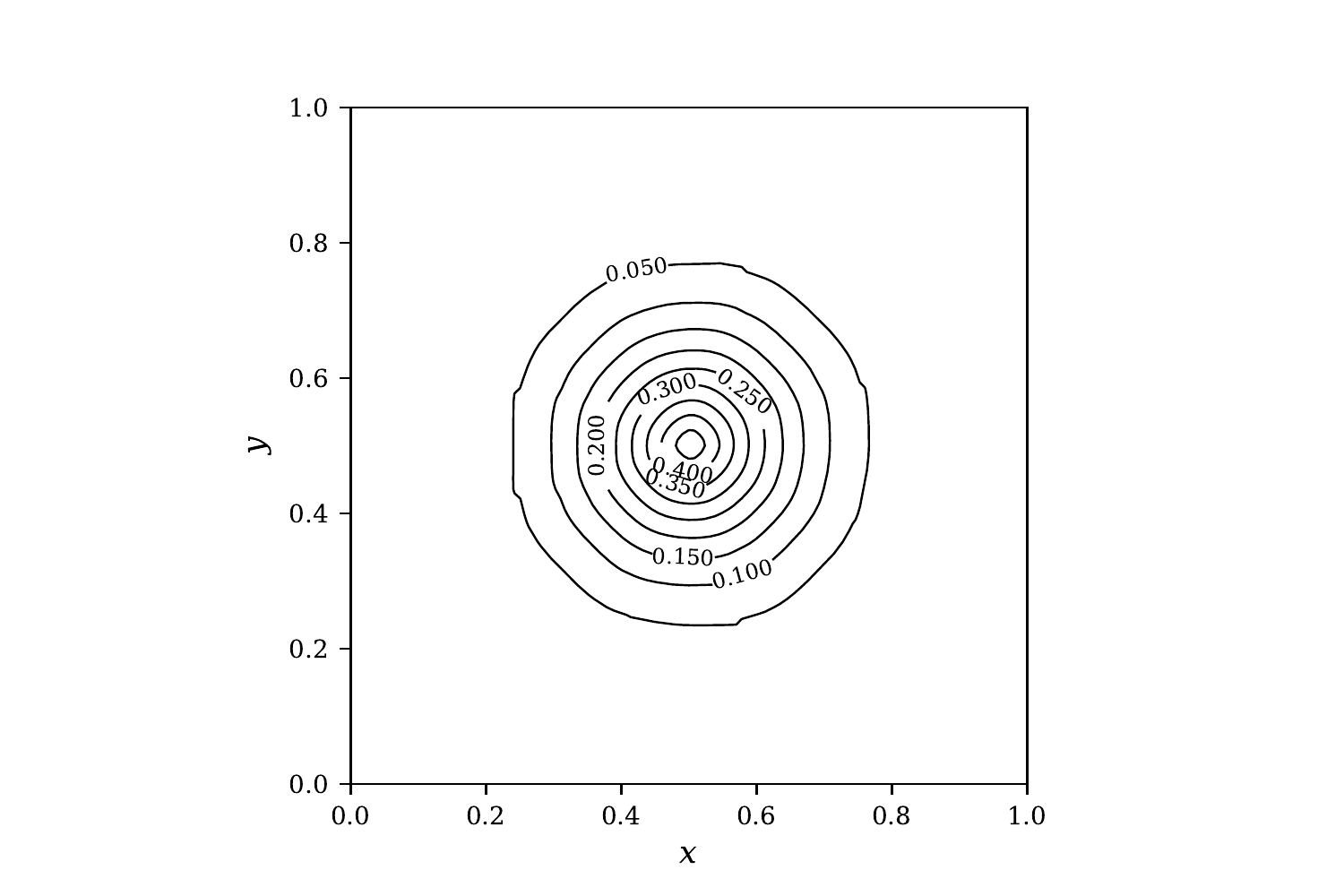}
         \caption{}
         \label{fig:gaussian-contours}
     \end{subfigure}
            \caption{Approximation of the Tukey depth using Equation \eqref{eqn:TukeyEqn}, for a normal distribution. Figure \ref{fig:gaussian-vals} shows the computed values, and Figure \ref{fig:gaussian-contours} shows the level sets of the computed values. Here the (square) grid spacing is $1/128$, and the right-hand side is computed by approximating the integral using a random sample (size 12k) from the domain.}
        \label{fig:gaussian}
\end{figure}

\begin{figure}
\centering
     \begin{subfigure}[b]{0.49\textwidth}
         \centering
         \includegraphics[width=\textwidth]{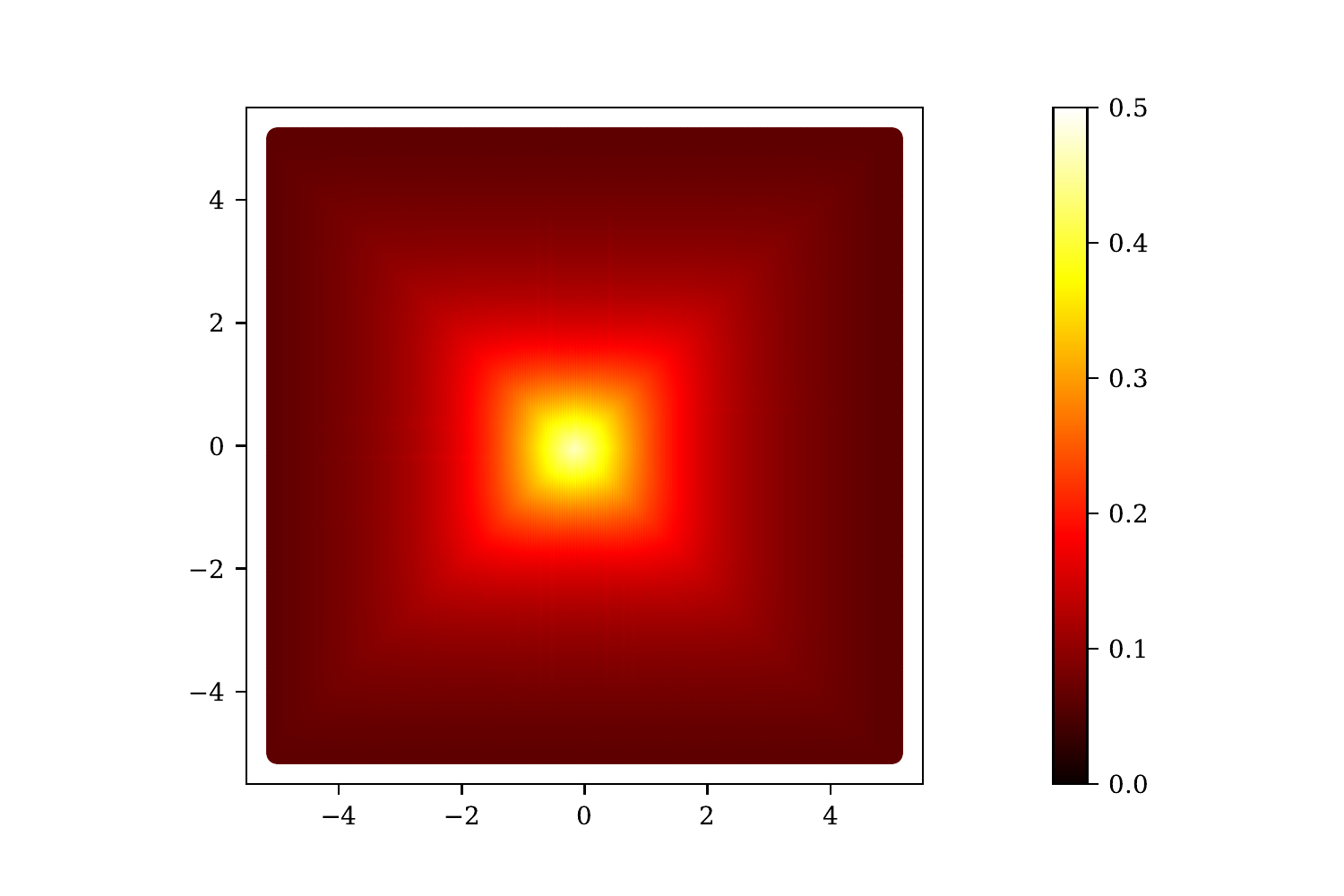}
         \caption{}
         \label{fig:cauchy-vals}
     \end{subfigure}
     \hfill
     \begin{subfigure}[b]{0.49\textwidth}
         \centering
         \includegraphics[width=\textwidth]{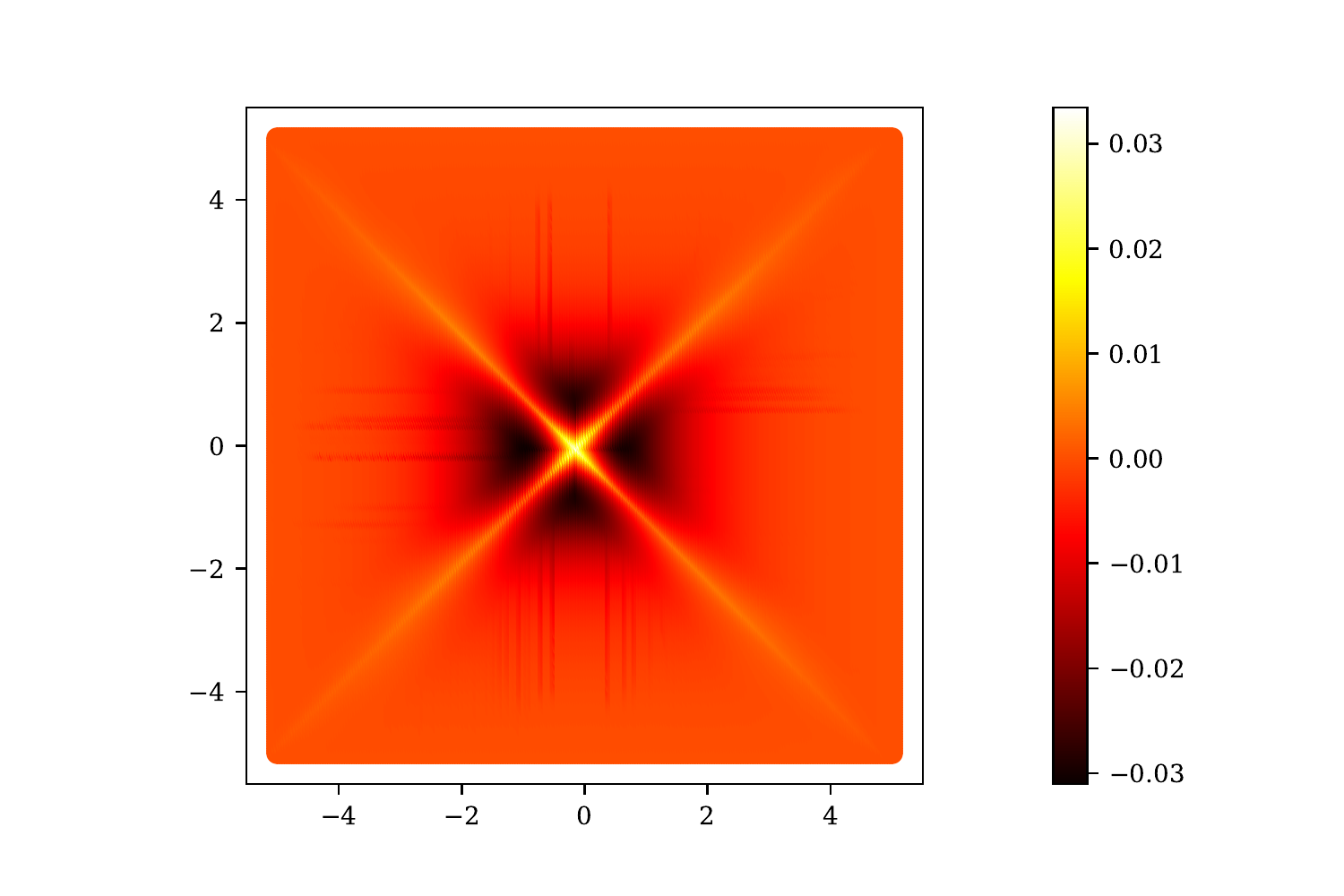}
         \caption{}
         \label{fig:cauchy-errors}
     \end{subfigure}
     \\
     \begin{subfigure}[b]{0.49\textwidth}
         \centering
         \includegraphics[width=\textwidth]{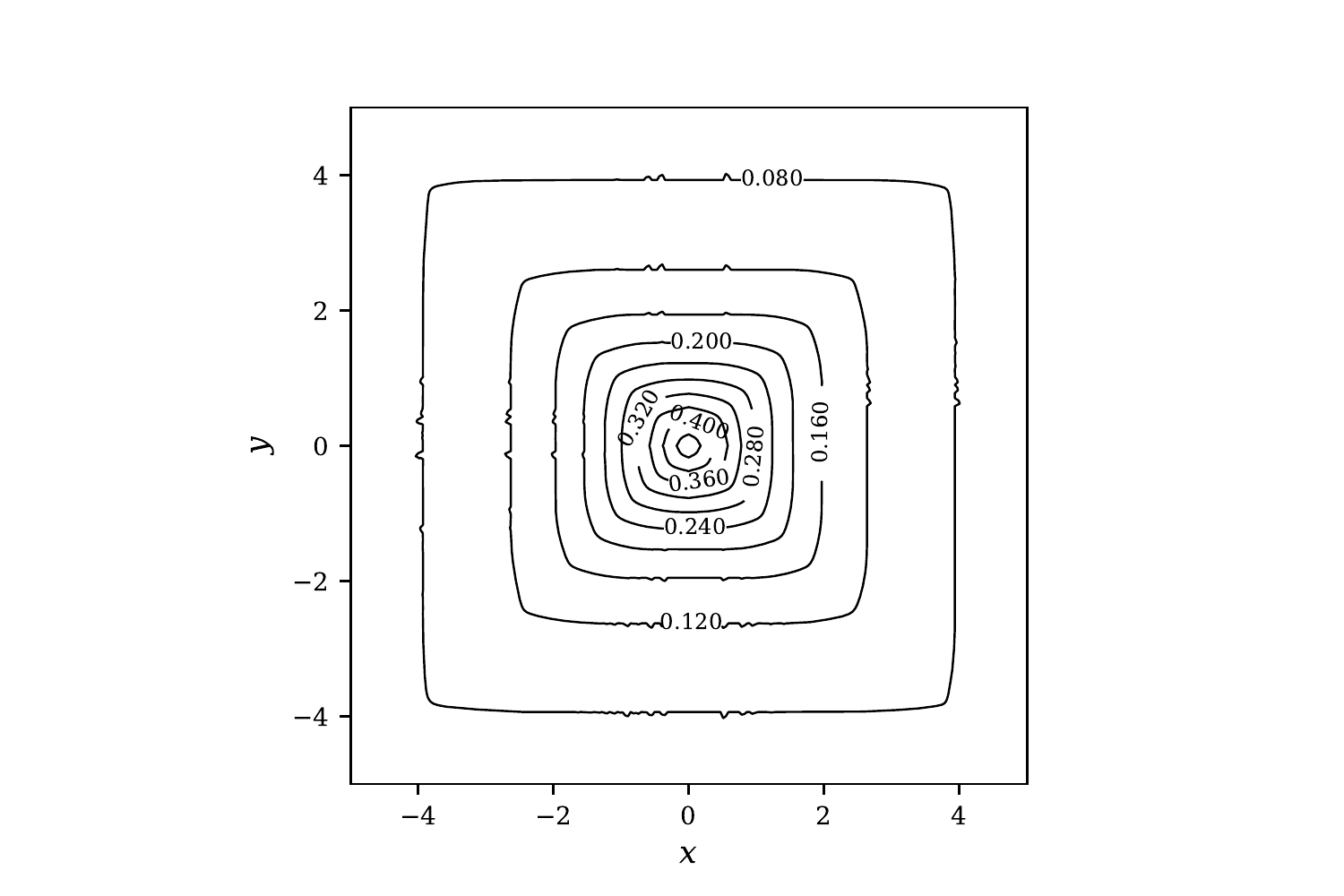}
         \caption{}
         \label{fig:cauchy-calc-contours}
     \end{subfigure}
     \hfill
     \begin{subfigure}[b]{0.49\textwidth}
         \centering
         \includegraphics[width=\textwidth]{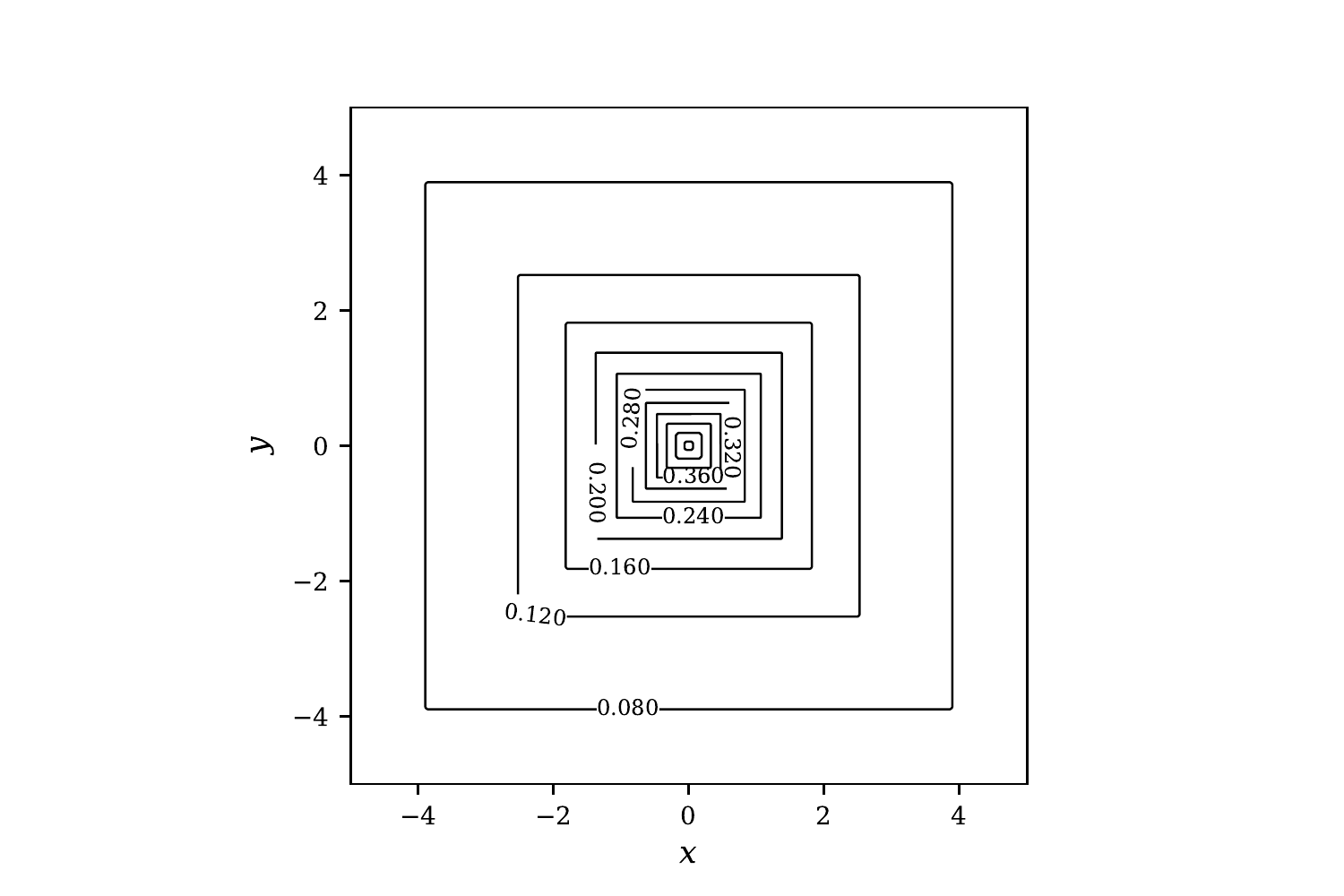}
         \caption{}
         \label{fig:cauchy-actual-contours}
     \end{subfigure}
     \\
     \begin{subfigure}[b]{0.49\textwidth}
     \centering
         \includegraphics[width=\textwidth]{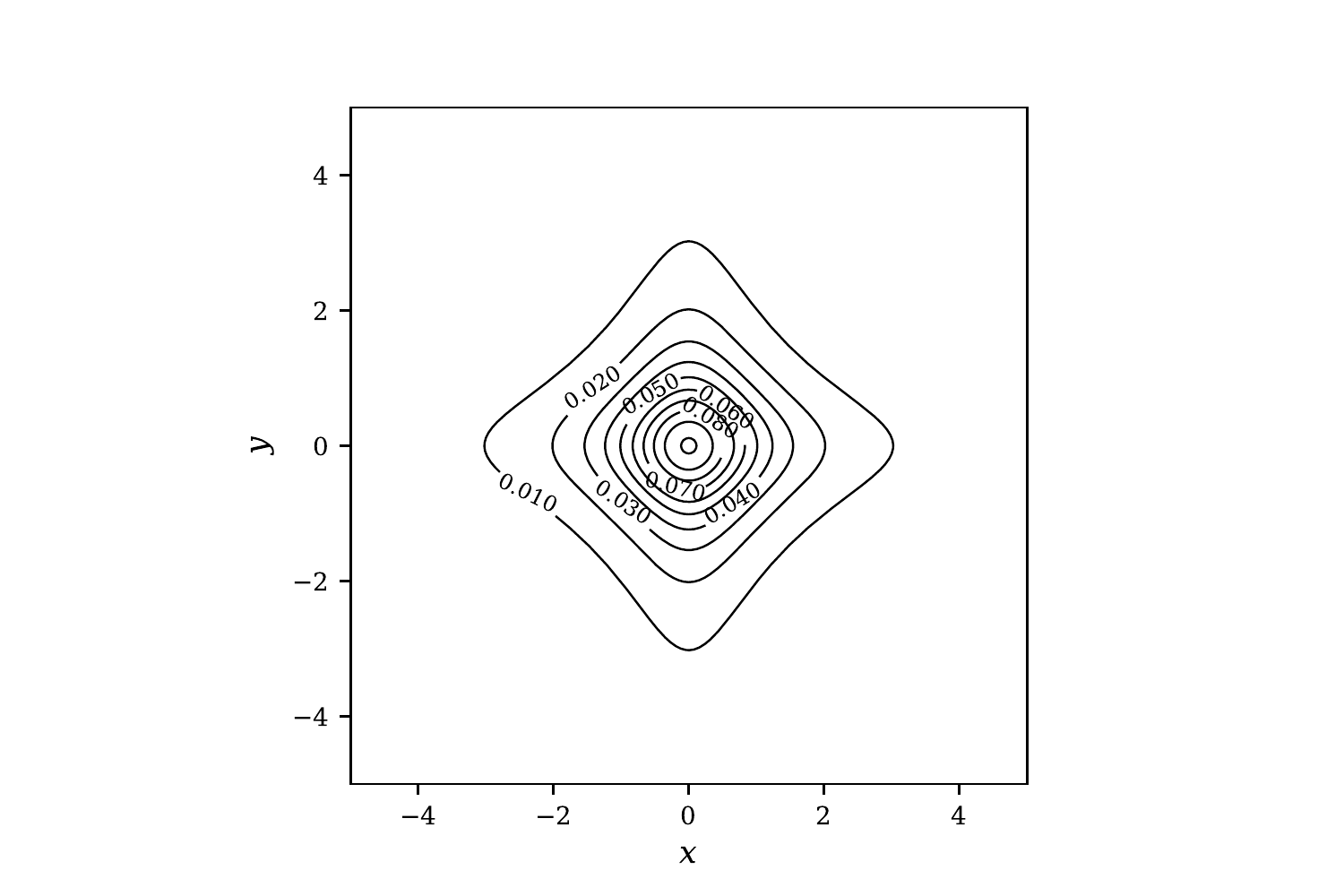}
         \caption{}
         \label{fig:cauchy-density-contours}
     \end{subfigure}
        \caption{Approximation of the Tukey depth using Equation \eqref{eqn:TukeyEqn}, for a bivariate Cauchy distribution with scale parameter 1. Figure \ref{fig:cauchy-vals} shows the computed values, Figure \ref{fig:cauchy-errors} shows the difference with the exact solution, Figure \ref{fig:cauchy-calc-contours} shows the level sets of the computed values and Figure \ref{fig:cauchy-actual-contours} shows the level sets of the true contours. Here the (square) grid spacing is $10/256$, and the right-hand side is computed by approximating the integral using a random sample (size 10k) from the domain. In Figure \ref{fig:cauchy-density-contours} we also provide the contours of the density function, which are not convex, as a point of reference.}
        \label{fig:cauchy}
\end{figure}

\clearpage
\bibliographystyle{plain}
\bibliography{mybib}

\end{document}